\documentclass{article}

\usepackage[utf8]{inputenc}
\usepackage{amsmath,amsfonts,amssymb}
\usepackage{amsthm,upref}
\usepackage[a4paper]{geometry}
\usepackage{color} 
\usepackage[dvipsnames]{xcolor}
\usepackage{graphicx}
\usepackage{subfigure}
\usepackage{tabularx,float}
\usepackage[hidelinks]{hyperref}
\usepackage{comment} 
\usepackage{tabularx} 
\usepackage{titlefoot} 
\usepackage{fancyhdr}

\usepackage{fullpage}

\usepackage{todonotes}
\usepackage{color} 

\newtheorem{thm}{Theorem}[section]

\newtheorem{proposition}[thm]{Proposition}

\newtheorem{lemma}[thm]{Lemma}
\newtheorem{remark}[thm]{Remark}

\newcommand{\eps}{\varepsilon}
\newcommand{\R}{\mathbb{R}}

\newcommand{\me}{\textrm{e}}

\newcommand\SJ[1]{{\color{black}{#1}}} 
\newcommand\SJJ[1]{{\color{black}{#1}}}
\newcommand\SJJJ[1]{{\color{black}{#1}}}
\newcommand\KUK[1]{{\color{black}{#1}}}
\newcommand\rsp[1]{{\color{black}{#1}}}
\newcommand\rspp[1]{{\color{black}{#1}}}

\title{Switching, Multiple Time-Scales and Geometric Blow-Up in a Low-Dimensional Gene Regulatory Network}
\author{S.~Jelbart, K.~U.~Kristiansen \& P.~Szmolyan}
\date{\today}

\begin{document}
	\maketitle
	
	\begin{abstract}
		ODE-based models for gene regulatory networks (GRNs) can often be formulated as smooth singular perturbation problems with multiple small parameters, some of which are related to time-scale separation, whereas others are related to `switching', (proximity to a non-smooth singular limit). This motivates the study of reduced models obtained after (i) quasi-steady state reduction (QSSR), which utilises the time-scale separation, and (ii) piecewise-smooth approximations, which reduce the nonlinearity of the model by viewing highly nonlinear sigmoidal terms as singular perturbations of step functions. We investigate the \textit{interplay} between the reduction methods (i)-(ii), in the context of a 4-dimensional GRN which has been used as a low-dimensional representative of an important class of (generally high-dimensional) GRN models in the literature. We begin by identifying a region in the small parameter plane for which this problem can be formulated as a smooth singularly perturbed system on a blown-up space, uniformly in the switching parameter. This allows us to apply Fenichel's coordinate-free theorems and obtain a rigorous reduction to a 2-dimensional system, that is a perturbation of the QSSR. Finally, we show that the reduced system features a Hopf bifurcation which does not appear in the QSSR system, due to the influence of higher order terms. Taken together, our findings suggest that the relative size of the small parameters is important for the validity of QSS reductions and the determination of qualitative dynamics in GRN models more generally. Although the focus is on the 4-dimensional GRN, our approach is applicable to higher dimensions.
	\end{abstract}
	
	\medskip
	
	

\section{Introduction}
\label{sec:introduction}

Mathematical models of biological and biochemical phenomena often feature \textit{switching}, i.e.,~abrupt dynamical transitions which occur when a particular variable or set of variables in the model cross over a threshold. 
In the context of ODE-based models, switching can often be described by smooth singular perturbation problems which limit to \textit{piecewise-smooth (PWS)} systems as one or more of the perturbation parameters tend to zero. Problems of this kind have received a lot of attention in recent years; here we refer to \cite{Jelbart2022b,Kosiuk2016,Kristiansen2019d,Miao2020,Jelbart2024,Mahdi2024} for recent work on low-dimensional problems featuring switching in ODE models of intracellular calcium oscillation, substrate-depletion oscillation, mitotic oscillation in a model for the embryonic cell cycle, a cAMP signalling model for cell aggregation and the Jansen-Rit model in mathematical neuroscience.

In addition to switching, many biological and biochemical processes -- and consequently also the mathematical models thereof -- evolve on multiple time-scales. This is also true of the particular applications mentioned above, and one of the main contributions of the references \cite{Jelbart2022b,Kosiuk2016,Kristiansen2019d,Miao2020,Jelbart2024} has been to demonstrate that switching and multiple time-scale aspects of the models can be analysed with a common set of mathematical techniques based on \textit{geometric singular perturbation theory (GSPT)} \cite{Fenichel1979,Jones1995,Kuehn2015,Wechselberger2019,Wiggins2013} and a method of desingularisation known as \textit{geometric blow-up} \cite{Dumortier1996,Kristiansen2015b,Krupa2001a,Krupa2001b,Szmolyan2001}. Singular perturbation analyses of problems involving both switching and multiple time-scales tend to be complicated by the presence of multiple small parameters, some of which relate to switching, whereas others are related to the multi-scale structure of the model. One approach to the analysis of these problems is to group the small parameters together by relating them to a single small parameter, as in e.g.~\cite{Jelbart2022b,Kruff2021,Jelbart2024}. This serves to simplify the analysis, however, approaches of this kind tend to rely on the availability of good numerical estimates for the parameter values, and one cannot in general expect the same dynamics if the singular limit is approached along a different curve in parameter space. Detailed mathematical analyses of problems with multiple small parameters that describe the asymptotic behaviour and singular limit in the corresponding small parameter space, where the small parameters are kept independent, are much more involved. Perhaps unsurprisingly, analyses of this kind appear to be less common in the literature, however we refer to \cite{Baumgartner2024,deMaesschalck2011,Kristiansen2024,Kuehn2022b} for notable exceptions. 

\subsection{Multiple time-scales and switching in GRNs}

This article is motivated by the interplay between multiple time-scale dynamics and switching in an important class of \textit{gene regulatory network (GRN)} models \SJ{which can be traced back to the pioneering work in \cite{Glass1973,Thomas1973}}. 
A significant body of research investigating the dynamics and properties of these systems has emerged in recent decades, see e.g.~\cite{Edwards2015,Machina2013a,Machina2013b,Polynikis2009,Quee2021} and the many references therein. The models in question can be written as $2N$-dimensional ODE systems of the form
\begin{equation}
	\label{eq:grn_network}
	\begin{split}
		r_i' &= F_{{i}}(h_{\SJJJ{i,1}}(p_1; \theta_1, n_1), h_{\SJJJ{i,2}}(p_2; \theta_2, n_2), \ldots, h_{\SJJJ{i,N}}(p_N; \theta_N, n_N)) - \gamma_i r_i , \\
		p_i' &= \eps \left( \kappa_i r_i - \delta_i p_i \right) ,
	\end{split}
\end{equation}
where $(\cdot)'=\frac{d}{dt}$, $i = 1, \ldots, N$, $r_i \geq 0$ denotes the concentration of transcribed mRNA associated with a particular gene $i$, $p_i \geq 0$ denotes the concentration of the corresponding translated protein, $\gamma_i, \eps, \kappa_i, \delta_i, \theta_i, n_i$ are strictly positive parameters, each $h_{\SJJJ{i,j}}$ is a Hill function of the form $h_{\SJJJ{i,j}}= h^+$ or $h_{\SJJJ{i,j}} = h^-$, where
\begin{align}\label{eq:hill0}
	h^-(\KUK{p}; \KUK{\theta}, \KUK{n}) = 1 - h^+(\KUK{p}; \KUK{\theta}, \KUK{n}) = 
	\frac{\KUK{\theta}^{\KUK{n}}}{\KUK{p}^{\KUK{n}} + \KUK{\theta}^{\KUK{n}}} ,
\end{align}
and {each} function $F_{{i}}$ is a polynomial in its arguments. We refer to e.g.~\cite{Edwards2015,Polynikis2009} for details, and to Figure \ref{fig:Hill_functions} for a graphical representation of the Hill functions $h^\pm$.

\begin{figure}[t!]
	\centering
	\includegraphics[scale=0.4]{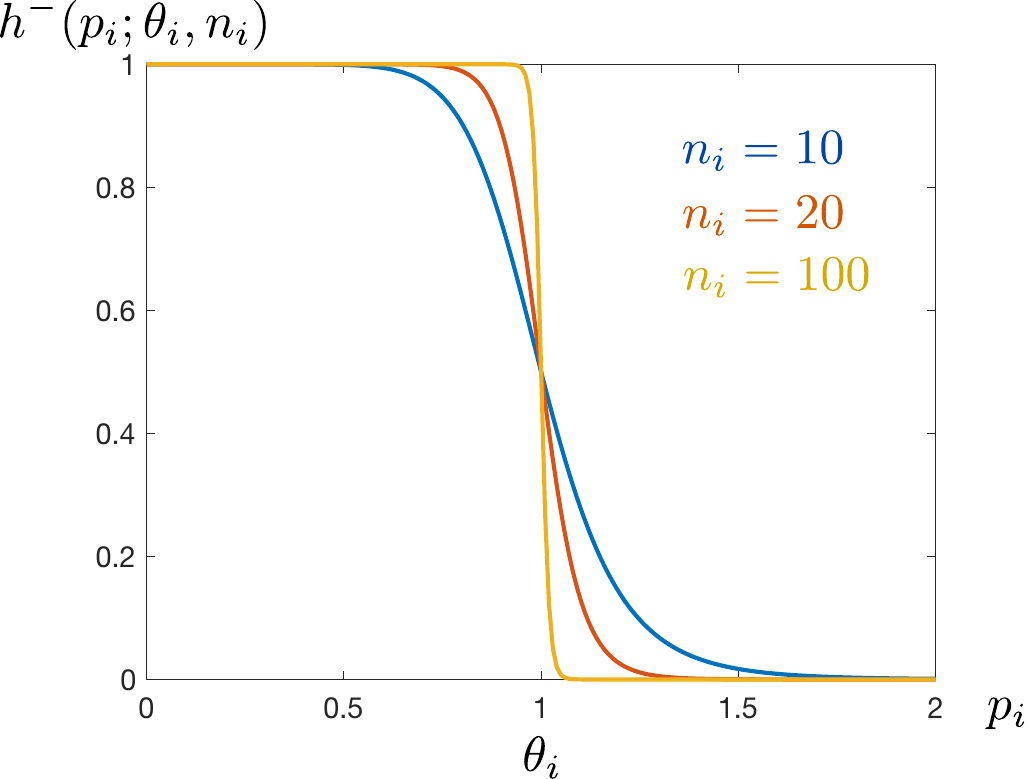}
	\quad
	\includegraphics[scale=0.4]{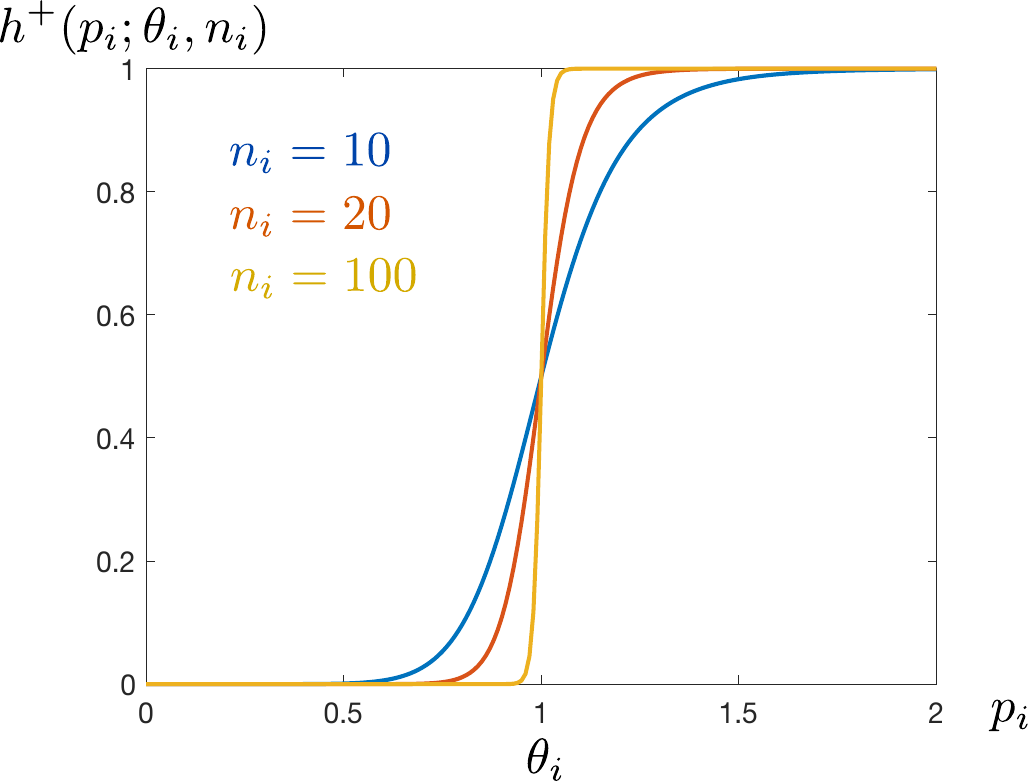}
	\caption{The Hill functions $h^-(p_i; \theta_i, n_i)$ and $h^+(p_i; \theta_i, n_i)$ are shown in in the left and right panels respectively, for $\theta_i = 1$ and three different values of $n_i$. Both functions converge do a discontinuous step function when $n_i \to \infty$, as in \eqref{eq:Hill_fns}.}
	\label{fig:Hill_functions}
\end{figure}

\begin{figure}[t!]
	\centering
	\includegraphics[scale=0.55,trim= 50 180 5 180]{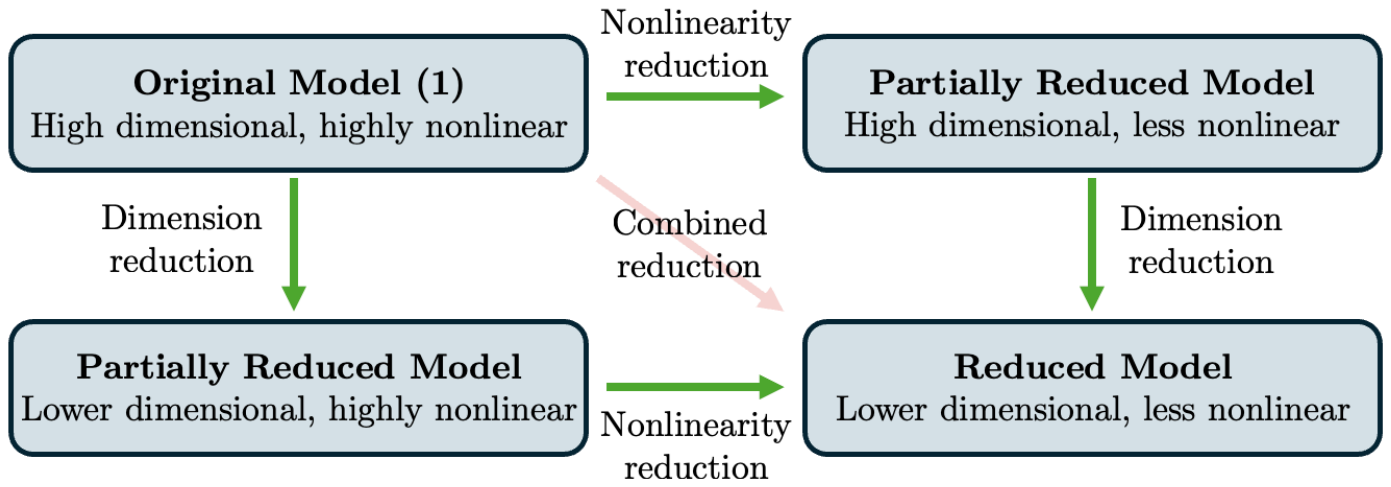}
	\caption{\rsp{Conceptual schema for the combined approach to model reduction in the (generally high dimensional and highly nonlinear) GRN model \eqref{eq:grn_network}. Applying partial reductions in different orders (or equivalently, taking a different path from the original model \eqref{eq:grn_network} to the final reduced model) can lead to different reduced models, and therefore to different predictions.}}
	\label{fig:model_reduction}
\end{figure}

\rsp{Direct analysis or simulation of the full network dynamics using system \eqref{eq:grn_network} is difficult because it is both (i) high dimensional and (ii) highly nonlinear. In order to deal with these obstacles in practice, mathematical modellers study simpler \textit{reduced models} which are lower dimensional, less nonlinear, or both. The conceptual schema is illustrated in Figure \ref{fig:model_reduction}. Notice that the high dimensionality and high nonlinearity of the original model \eqref{eq:grn_network} are dealt with via separate reductions, referred to as \textit{partial reductions} in the figure and its caption. Depending on the context, it may be sufficient to work with a partially reduced model which is `only' lower-dimensional or less nonlinear, but in many cases the overall complexity of the problem requires a lower-dimensional, less nonlinear reduced model is only obtained after applying both of these reductions together; we refer to \cite{Polynikis2009} for a review and investigation into the validity of the reductions that are often applied in practice. This paper will focus on the utility and validity of \textit{combined reductions} of this kind.} 

\rsp{In mathematical terms, the partial and independent dimensional/nonlinearity reductions described above are related to the presence of independent small parameters in system \eqref{eq:grn_network}. The main assumptions are:} 
\begin{itemize}
\item[(i)] \textbf{QSSR}: $0 < \eps \ll 1$. This amounts to the assumption that mRNA transcription is very fast compared with the protein dynamics. Applying the QSSR\SJJJ{, i.e.~setting $r_i' = 0$ for each $i = 1, \ldots, N$,} reduces the $2N$-dimensional network \eqref{eq:grn_network} to an $N$-dimensional \textit{protein-only network}
\[
\dot p_i = \frac{\kappa_i}{\gamma_i} F_{{i}}(h_\SJJJ{i,1}(p_1; \theta_1, n_1), h_\SJJJ{i,2}(p_2; \theta_2, n_2), \ldots, h_\SJJJ{i,N}(p_N; \theta_N, n_N)) - \delta_i p_i , \ \ 
i = 1, \ldots, N ,
\]
with
\[r_i = \gamma_i^{-1} F_{{i}}(h_\SJJJ{i,1}(p_1; \theta_1, n_1), h_\SJJJ{i,2}(p_2; \theta_2, n_2), \ldots, h_\SJJJ{i,N}(p_N; \theta_N, n_N)),\quad i=1,\ldots,N.\]
Here the overdot denotes differentiation with respect to slow time $\tau=\eps t$. 
\item[(ii)] \textbf{\SJJJ{Steep switching}}: the Hill exponents satisfy $n_i \gg 1$, motivating a PWS approximation in which the Hill functions $h^\pm$ \SJJJ{(which are very steep when $n_i \gg 1$)} are replaced by piecewise-constant functions in accordance with
\begin{equation}
\label{eq:Hill_fns}
\lim_{n_i \to \infty} h^+(p_i; \theta_i, n_i) = 
\begin{cases}
0 & p_i < \theta_i , \\
1 & p_i > \theta_i ,
\end{cases}
\qquad
\lim_{n_i \to \infty} h^-(p_i; \theta_i, n_i) = 
\begin{cases}
1 & p_i < \theta_i , \\
0 & p_i > \theta_i .
\end{cases}
\end{equation}
This (pointwise) convergence is sketched for increasing values of $n_i$ in Figure \ref{fig:Hill_functions}.
\end{itemize}
\rsp{The assumptions} (i) and (ii) simplify the analysis and/or simulation of the original network \eqref{eq:grn_network}, although in very different ways. \rsp{On one hand, the QSSR in (i) leads to a dimensional reduction which} takes advantage of the multi-scale (specifically slow-fast) structure of system \eqref{eq:grn_network} when $0 < \eps \ll 1$, in order to reduce the the dimension from $2N$ to $N$. On the other hand, (ii) \rsp{leads to a nonlinearity reduction which} reduces system \eqref{eq:grn_network} to a \textit{piecewise-linear (PWL)} system. Although the resulting equations in this case are non-smooth, they are \textit{linear} away from $p_i = \theta_i$.

The assumption in (i) is very common in the literature, and often assumed implicitly, however we refer to \cite{Edwards2015,Gedeon2012,Polynikis2009} for studies which involve the mRNA variables. (It is \SJ{also} important to note that this assumption may not always be justified; see \cite{Garcia2004,Gedeon2012} (p.~10,958 second paragraph in the former) for experimental values in actual GRN's, which demonstrate that the time-scale separation between the $r_i$ and $p_i$ variables varies for different applications.) The authors in \cite{Edwards2015} in particular proved a number of reduction theorems in which the QSSR can be made rigorous via a slow manifold reduction based on Fenichel theory. These results are valid for \rsp{fixed $n_i > 0$ as $\eps \to 0$, i.e.~they hold for the partially reduced model shown in the bottom left corner of Figure \ref{fig:model_reduction}}. In contrast, with this paper we begin a research program where we are interested in the \SJJ{combined} limit $n_i\to \infty$, $\eps\to 0$ and on the \textit{interplay} between the reduction assumptions in (i) and (ii) above. More broadly, we are interested in whether -- and if so how -- the qualitative dynamics depends on the relative sizes of the small parameters associated with switching and multi-scale structure. 

Numerous approaches to the non-smooth limit in (ii) appear in the literature: see \cite{Casey2006,Gouze2002} for an analyses based on Filippov theory \cite{Filippov1960}, and \cite{Ironi2011,Plahte2005} for foundational work on approaches based on singular perturbation methods. 
Both approaches have their pros and cons. On one hand, approaches based on Filippov theory are `consistent' on the PWS level, however, the dynamics prescribed at discontinuity sets may not accurately represent the dynamics of the original network \eqref{eq:main0}, which is smooth. On the other hand, approaches based on singular perturbation approach can sometimes be `restrictive' \cite{Edwards2015}, simply due to the complicated nature of the relationship between the original network and the non-smooth limit. We refer to \cite{Machina2013a,Machina2013b,Machina2011} for more on the relationship between these two approaches, and additional work from a singular perturbations standpoint. For the most part, the analyses presented in these works describe the limiting dynamics in reduced protein-only systems. 
In other words, the QSSR is applied first (or it is implicitly assumed and the analysis starts with a protein-only system), after which an asymptotic analysis as the Hill coefficients $n_i \to \infty$ is undertaken. \rsp{In addition to the references above, we} refer to \cite{Gameiro2024} and the many references therein for details on an \SJJJ{interesting} approach to the study of global dynamics of switching systems based on combinatorial, topological and computational methods, which has direct significance for the study of GRNs.

\rsp{In terms of Figure \ref{fig:model_reduction}, the aforementioned results for the nonlinearity reduction in (ii) apply to the reduction step which is oriented from the partially reduced problem on the bottom left to the (fully) reduced problem on the bottom right (since they assume the validity of the QSSR as a preliminary step). It is, however, worthy to emphasise that the final reduced model obtained via this procedure will \textit{not}, in general, be the same as the reduced model that is obtained by taking the `other path' which proceeds via the top right partially reduced model in Figure \ref{fig:model_reduction}. This follows from a subtle but important fact about problems with multiple small parameters, namely, that the limiting dynamics depends upon the way in which the simultaneous singular limit is taken. In the context of the GRN model \eqref{eq:grn_network}, this leads to non-unique reduced models and therefore to an important problem in practice: how to determine the \textit{correct} reduced model.}

\subsection{\rspp{A representative} model}
\label{sub:model}


\rsp{Our aim in what follows is to investigate the questions outlined above in detail for a simple, $4$-dimensional activator-inhibitor model in general form \eqref{eq:grn_network}, which has previously been used as a `representative' model in model reduction investigations in \cite{Polynikis2009} (we also refer to \cite{Widder2007} for an earlier analysis of the model). The equations can be formulated as follows:}
\begin{equation}
\label{eq:main0}
\begin{aligned}
r_a' &= h^+(p_b;1,\sigma^{-1}) -   r_a , \\
r_b' &= h^-(p_a;1,\sigma^{-1}) - \gamma  r_b , \\
p_a' &=  \eps \left( \xi_a  r_a - p_a \right) , \\
p_b' &=  \eps \delta \left( \xi_b  r_b -   p_b \right) .
\end{aligned}
\end{equation}
\SJ{System} \eqref{eq:main0} is in the general form \eqref{eq:grn_network} 
with two Hill functions appearing on the right-hand side with $\theta_a=\theta_b=1$ and
\begin{align}\label{eq:sigman1n2}
n_a=n_b= \sigma^{-1},
\end{align}
so that $\sigma \to 0$ corresponds to the switching limit $n_a = n_b \to \infty$.
We consider $\eps, \sigma>0$ as independent small parameters, and $\delta,\gamma,\xi_i,\,i\in \{a,b\}$ all positive and fixed in some compact set; in particular, we will treat $\xi_a$ and $\xi_b$ as bifurcation parameters (fixed in some compact set). 

\begin{figure}[t!]
\centering
\subfigure[]{\includegraphics[scale=0.45]{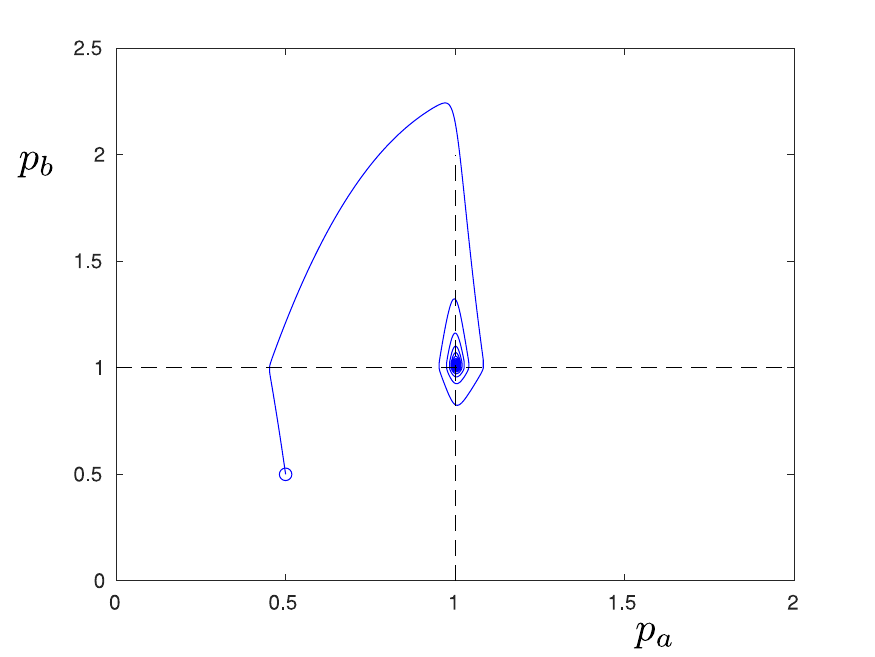} }  
\subfigure[]{\includegraphics[scale=0.45]{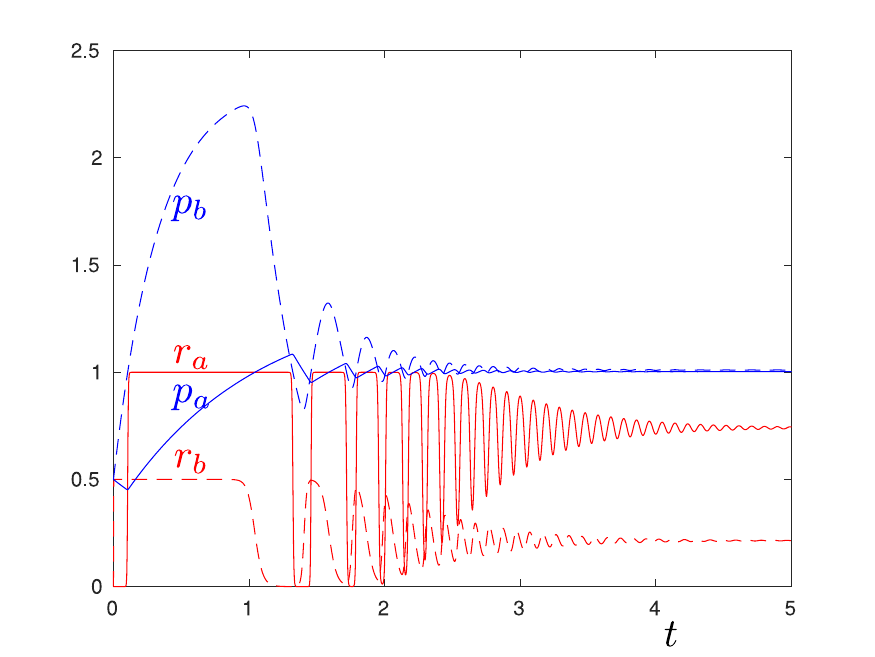}  } 
\subfigure[]{\includegraphics[scale=0.45]{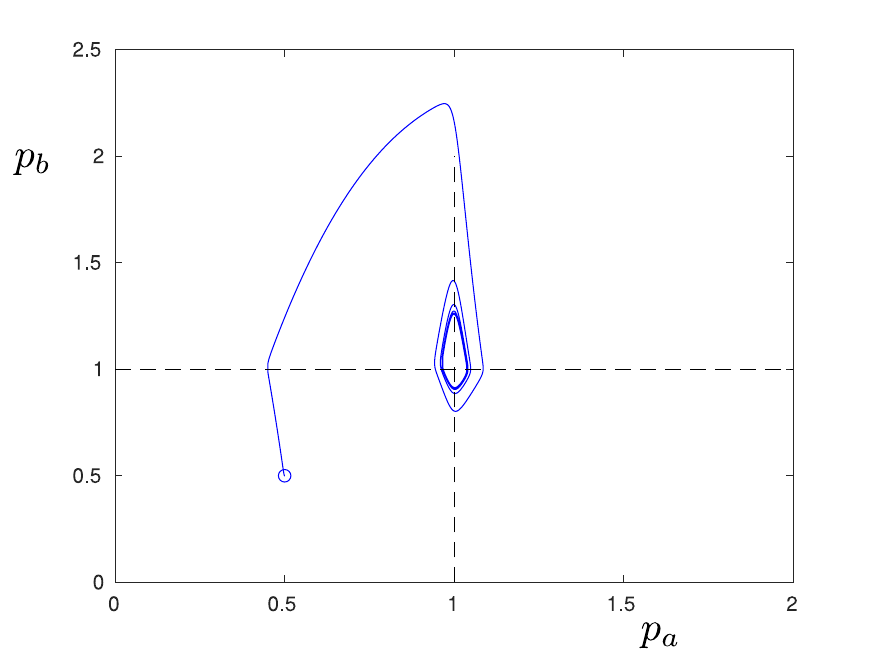} } 
\subfigure[]{\includegraphics[scale=0.45]{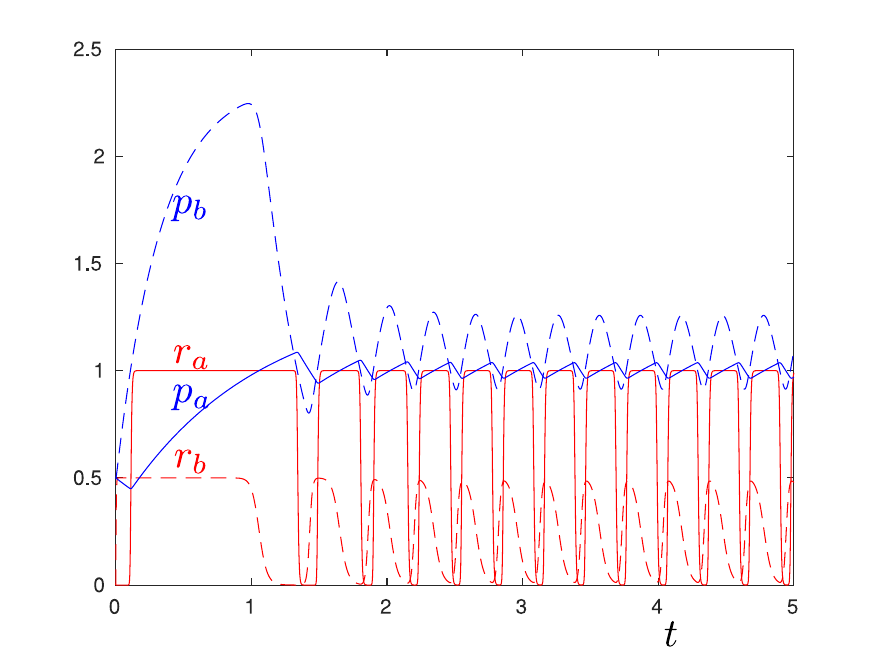} } 
\caption{\rsp{Two different particular solutions of system \eqref{eq:main0}. In (a) and (c): Projections onto the $(p_a,p_b)$-plane. In (b) and (d): Time series of solutions. The parameter values are given by \eqref{eq:para} and $\eps=5\times 10^{-5}$ in (a) and (b) and $\eps=5\times 10^{-3}$ in (c) and (d). The presence of oscillations appears to depend upon $\eps$.}}
\label{fig:lcnolc}
\end{figure}

The notation and formulation in system \eqref{eq:main0} agrees with \cite[eqn.~(41)]{Polynikis2009}, except for the fact that we have assumed \eqref{eq:sigman1n2} and introduced certain scalings to eliminate extra parameters; we refer to Appendix \ref{app:model} below for details on this derivation. The reader is also referred to \cite{Widder2007} \SJ{(where to the best of our knowledge, the system first appeared)} for \SJ{an earlier analysis} \SJJ{and} biochemical interpretation of \SJJ{system} \eqref{eq:main0}. \SJJ{We also} refer to 
\cite{delVecchio2007,Edelstein1988,Guantes2006} for analyses of similar systems that involve additional terms which model self-regulation effects.'

\rsp{An interesting aspect of \eqref{eq:main0} is that the existence of limit cycles appears to depend on $\eps$. We illustrate this in Figure \ref{fig:lcnolc}. Here we shown the result of numerical computations of solutions of \eqref{eq:main0} for the parameter values:
\begin{align}
\gamma=2,\ \,\delta = 3,\ \, \xi_a = 1.3536,\ \,\xi_b = 2.3536,\ \,\sigma=10^{-2},\label{eq:para}
\end{align}
and two different values of $\eps$. In (a) and (b), we have $\eps=5 \times 10^{-5}$, whereas in (c) and (d), we have used $\eps=5 \times 10^{-3}$. The panels (a) and (c) show projections to the $(p_a,p_b)$-plane whereas (c) and (d) show the time series of $(r_a,r_b,p_a,p_b)$. All figures use the same initial condition $(r_a,r_b,p_a,p_b)(0)=(0,0,0.5,0.5)$ and the solution is computed using Matlab's ODE15s with low tolerances ($10^{-9}$). We see that in (a) and (b), the solution converges to a steady state, whereas in (c) and (d) the solution converges to a periodic orbit.  }


\rsp{Indeed, earlier work -- see \cite{Tyson1978,Edelstein1988} for reviews and \cite{Polynikis2009} for a result on system \eqref{eq:main0} in particular -- shows mathematically that QSSR `destroys' the oscillations in the limit $\eps \to 0$ (assuming $\sigma > 0$ is fixed). In the case of system \eqref{eq:main0},} 
applying a QSSR to system \eqref{eq:main0} 
leads to \SJJJ{the} 2-dimensional protein-only system
\begin{equation}\label{eq:M0red}
\begin{aligned}
\dot p_a &= \xi_a h^+(p_b;1,\sigma^{-1}) -p_a,\\
\dot p_b &= \delta (\xi_b \gamma^{-1} h^-(p_a;1,\sigma^{-1})-p_b),
\end{aligned}
\end{equation}
with $r_a$ and $r_b$ slaved according to
\begin{equation}\label{eq:M0}
r_a = h^+(p_b;1,\sigma^{-1}),\qquad
r_b = \gamma^{-1} h^-(p_a;1,\sigma^{-1}).
\end{equation}
Since the divergence of \SJJ{the vector field associated with system} \eqref{eq:M0red} is $-1-\delta<0$ for all $\sigma>0$, we conclude the following:
\begin{lemma}\label{lemma:qssr}
The QSSR \SJJ{system} \eqref{eq:M0red} does not support limit cycles. 
\end{lemma}
\rsp{
Lemma \ref{lemma:qssr} is significant because it means that the QSSR fails to provide an accurate representation of the original system dynamics. In particular, the reduced problem \eqref{eq:M0red} cannot be used to explain the limit cycles observed in Figure \ref{fig:lcnolc} (c) and (d). It would seem, then, that the reduction to a protein-only system fails. However, the preceding analysis \textit{assumed implicitly that $\sigma > 0$ is kept fixed}. As we will see below, system \eqref{eq:main0} \textit{can} be reduced to a protein-only system which can oscillate for arbitrarily small values of $\eps$, as long as the ratio $\eps / \sigma$ is small. In order to show this, we will need to consider the QSSR more carefully.}

\SJJJ{The reduction of system \eqref{eq:main0} to the reduced protein-only system \eqref{eq:M0red} can be justified using Fenichel theory \cite{Fenichel1979,Jones1995}, since} \eqref{eq:M0} defines a critical manifold $\mathcal M(\sigma)$ of \eqref{eq:main0} for $\eps=0$ and any $\sigma>0$. It is attracting and normally hyperbolic, since the linearization of \eqref{eq:M0} gives $-1$ and $-\gamma^{-1}$ as the nonzero eigenvalues. Consequently, Fenichel theory provides the existence of a slow manifold for all $0<\eps\ll  1$\SJJJ{, and the dynamics on the slow manifold is a smooth perturbation of the QSSR system \eqref{eq:M0red}}. \rsp{Indeed, an analogous argument has been used to mathematically justify the QSSR in the general network \eqref{eq:grn_network} as long as each $n_i > 0$ is fixed \cite{Edwards2015}.} \SJJ{However}, this result is not uniform with respect to $\sigma \to 0$ \rsp{(or $n_i \to \infty$ in the context of \eqref{eq:grn_network}). For system \eqref{eq:main0}, this is} due to loss of smoothness along the codimension-1 \textit{switching manifold} 
\begin{align}\label{eq:Sigma}
\Sigma := \Sigma_a \cup \Sigma_b,
\end{align}
where
\begin{align}\label{eq:Sigmaab}
\Sigma_a = \left\{ (r_a, r_b, 1, p_b) : r_a, r_b, p_b \geq 0 \right\} , \qquad 
\Sigma_b = \left\{ (r_a, r_b, p_a, 1) : r_a, r_b, p_a \geq 0 \right\},
\end{align}
recall \eqref{eq:Hill_fns}. \SJJJ{This indicates that the double limit as both $\eps \to 0$ and $\sigma \to 0$ needs a \rspp{more} careful analysis.}

\subsection{Main results}
\SJJ{Our main results \SJJJ{establish the existence of} a region in the $(\sigma, \eps)$ small parameter plane for which the 
reduction to a 2-dimensional protein-only system is possible. They also describe certain aspects of the dynamics of this reduced problem, which should be compared with the QSSR system \eqref{eq:M0red}}. 
\begin{thm}\label{thm:thm0}
Consider \SJJ{system} \eqref{eq:main0} with
\begin{equation}
\label{eq:mu_scaling}
\eps = \mu \sigma,
\end{equation} and fix any $k\in \mathbb N$, $\sigma_0>0$ sufficiently small and any \SJJ{connected} compact domain $K\subset [0,\infty)^2$. Then the following \SJJJ{assertions are true:} 
\begin{enumerate}
\item \label{existence}
\rsp{Reduction to a protein-only system is justified when $0 < \mu \ll 1$. More precisely,} \SJJJ{there exists a $\mu_0 > 0$ small enough such that t}here exists a \SJJJ{2-dimensional} attracting \SJ{and locally invariant} manifold $\mathcal M_{\mu}(\sigma)$ for all $\mu \in (0,\mu_0)$, $\sigma\in (0,\sigma_0)$, 
with $\mu_0$ independent of $\sigma$, which is
\begin{enumerate}
\item a $C^k$-smooth graph
over $(p_a,p_b)\in K$, depending smoothly upon $\sigma\in (0,\sigma_0)$ and $\mu \in[0,\mu_0)$;
\item $C^0$ $\mathcal O(\mu)$-close to the critical manifold $\mathcal M_0(\sigma)\subset \mathcal M(\sigma)$ defined by the QSSR approximation in \eqref{eq:M0}, uniformly with respect to $\sigma$. 
\end{enumerate}
\item \label{kappa}
\rsp{Hopf bifurcation is only possible in a particular region of the $(\eps, \sigma)$-plane.} \SJ{Let}
\begin{align}\label{eq:alpha}
\alpha:=\frac{\gamma (1+\delta)}{\delta(1+\gamma)}.
\end{align}Then 
the ray defined by
\begin{align*}
\mu = \alpha \sigma,\qquad \sigma>0,
\end{align*}
divides the $(\sigma,\mu)$-parameter plane into two separate regions, where Hopf bifurcations on $\mathcal M_{\mu}(\sigma)$ (as $(\sigma,\mu)\rightarrow \SJJ{(0,0)}$) only occur (for some values of our bifurcation parameters $(\xi_a,\xi_b)$) within $\{(\sigma,\mu)\,:\,\mu>\alpha \sigma\}$.
\item \label{nonexistence}
\rsp{Reduction to a protein-only system is not possible when $\mu \sim 1$. More precisely:} \SJJJ{Fix} $\mu>0$ and consider $\sigma\to 0$. \SJJJ{There exists} a domain $K'$ \SJJJ{on which there does not exist a 2-dimensional} locally invariant manifold \SJJJ{of system \eqref{eq:main0}} which is $C^0$ $o(1)$-close (uniformly over the domain $K'$) to \SJJJ{the QSSR manifold $\mathcal M(\sigma)$} as $\sigma\to 0$.
%
\end{enumerate}
\end{thm}

Figure \ref{fig:par_planes} shows a schematic representation of Theorem~\ref{thm:thm0} in the $(\sigma, \eps)$ small parameter plane. Here, the invariant manifold $\mathcal M_{\mu}(\sigma)$ (see assertion \ref{existence}) exists in the triangular region bounded below the line $\eps = \mu_0 \sigma$. Figure \ref{fig:par_planes} also showcases assertion \ref{kappa}; notice that the line $\mu=\alpha\sigma$ becomes $\epsilon=\alpha \sigma^2$ in the $(\sigma,\eps)$-plane. This curve distinguishes (asymptotically as $(\sigma,\eps)\to \SJJ{(0, 0)}$) subregions within which it is possible or impossible for the system to undergo a Hopf bifurcation under variation of $\xi_a$ and $\xi_b$; these are the regions in shaded red and blue respectively. We will present a more detailed statement regarding the curve $\eps =\alpha\sigma^2$ below \SJJ{in} Proposition \ref{prop:kappa}.  

The fact that the reduced system on $\mathcal M_{\mu}(\sigma)$ -- which depends upon higher order corrections that are not accounted for in a QSSR, see \eqref{eq:M0red} --  may or may not undergo a Hopf-type bifurcation depending on the size of $\mu$ (cf. assertion \ref{kappa} of Theorem~\ref{thm:thm0}), shows that \SJJ{equation} \eqref{eq:mu_scaling} with $0<\mu\ll 1$ is necessary but not sufficient for the validity of the QSSR (loosely speaking, we consider a QSSR to be `valid' if it predicts the same qualitative dynamics as the original non-reduced problem). In particular, \SJJ{our results show} that QSSR is \SJJ{\textit{not valid}} in the red region of Figure \ref{fig:par_planes}, cf.~Lemma~\ref{lemma:qssr}\SJJ{, even though $0 < \mu \ll 1$}.

\begin{figure}[t!]
\centering
\quad \quad
\includegraphics[scale=0.4]{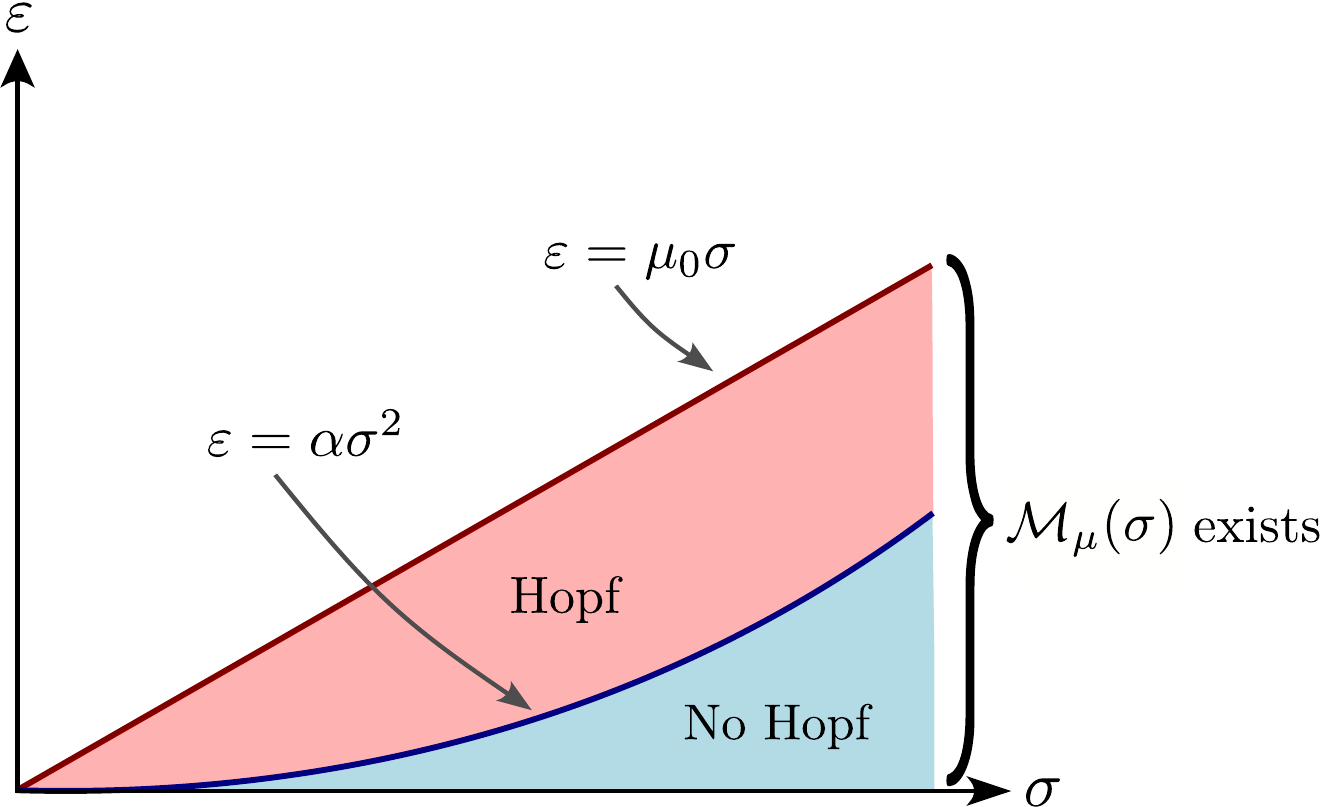}
\caption{The positive quadrant of the $(\sigma,\eps)$-plane is shown. The line $\eps = \mu_0 \sigma$, i.e.~$\mu = \mu_0$, defines a boundary below which our first main result (Theorem \ref{thm:thm0} assertion \ref{existence}) guarantees the existence of $\mathcal M_{\mu}(\sigma)$. The  quadratic curve $\eps = \alpha \sigma^2$ in dark blue divides the region defined by $0<\eps < \mu_0 \sigma$ into two distinct (asymptotic) subregions. Hopf bifurcations are possible (impossible) in the red (blue) subregion; see assertion \ref{kappa} of Theorem \ref{thm:thm0} and the discussion in the text for details.}
\label{fig:par_planes}
\end{figure}

\begin{remark}
\SJ{The manifold $\mathcal M_\mu(\sigma)$ can be treated like a `global' invariant manifold in many cases, since it is forward invariant on suitably chosen but arbitrarily large choices of the compact set $K$. For example, if $K = [0,A] \times [0,B]$ for any arbitrarily large but fixed $A, B > 0$ satisfying $A > \xi_a$ and $B > \xi_b \gamma^{-1}$, then $\mu_0 > 0$ can be chosen sufficiently small that the manifold $\mathcal M_\mu(\sigma)$ is forward invariant as a graph over $K$ for all $0 < \mu < \mu_0$.}
\end{remark}

\begin{remark}\label{rem:nonexistence}
\SJJ{Assertion} \ref{nonexistence} of Theorem~\ref{thm:thm0} \SJJ{shows that system} \eqref{eq:main0} with $\eps=\mu \sigma$, $\mu>0$ fixed, $0<\sigma\ll 1$, cannot in general be reduced to a protein-only system that remains $o(1)$-close to \SJJ{the} QSSR over the entire physically relevant domain. \SJJ{However,} it is still possible to perform (rigorous) local \SJ{reductions} as long as we stay uniformly away from $\Sigma$. This is also true for the general \SJJ{network} \eqref{eq:grn_network} (with $\Sigma=\cup_{i=1}^N \{p_i=\theta_i\}$) and it is a consequence of the fact that the limits \eqref{eq:Hill_fns} are uniform with respect to $p_i\ge 0$ on compact intervals $[a,b]$, $0\le a<b$, with $\theta_i\notin [a,b]$.
\end{remark}

\begin{remark}
\SJ{\rsp{Recall that the authors in \cite{Polynikis2009}} showed that the Hopf-type bifurcation observed in system \eqref{eq:main0} `disappears' after QSSR, when $\sigma > 0$ is kept fixed and positive \rsp{(we refer back Lemma \ref{lemma:qssr} and Figure \ref{fig:lcnolc})}. This is consistent with Theorem \ref{thm:thm0}, and can be understood in terms of Figure \ref{fig:par_planes}: taking $\eps \to 0$ along any fixed line $\sigma > 0$ will lead to a limiting system which corresponds to a point on the positive $\sigma$-axis, which is contained within the shaded blue region where Hopf bifurcation is not possible. 
Notice, however, that Hopf bifurcation is still possible as $\eps \to 0$ as long as (i) $\sigma \to 0$ as well and (ii) $(\sigma, \eps) = (0,0)$ is approached from the interior of the red region. \rsp{It follows that system \eqref{eq:main0} can oscillate for arbitrarily small values of $\eps$, as long as the ratio $\mu = \eps / \sigma$ is of a suitable magnitude.}

More generally, Figure \ref{fig:par_planes} shows that the qualitative dynamics of system \eqref{eq:main} as $\SJJ{(\sigma, \eps)} \to (0,0)$ depends in an important way on \textit{how} the point $(0,0)$ is approached, i.e.~on the path in $\SJJ{(\sigma, \eps)}$-space along which the limit is taken. This is a subtle but important point, because it shows that the relative size of the small parameters $\eps$ and $\sigma$ (and not only the fact that both $\eps \to 0$ and $\sigma \to 0$) plays an important role in determining the qualitative dynamics of the model.}
\end{remark}

\rsp{We now provide a brief overview of the proof, deferring the details to Section \ref{sec:proofs} and Appendix \ref{app:blow-up}. In order to resolve the loss of smoothness along the switching manifold $\Sigma$ when $\sigma = 0$, we} follow \cite[Section 8]{kaklamanos2019a}, which uses blow-up to analyse regularized PWS systems with \SJ{intersecting} discontinuity sets. We also refer to \cite{Kristiansen2019c,Kristiansen2015b,Kristiansen2015a,Kristiansen2019,Llibre2009} for earlier work on the use of blow-up for regularized PWS systems and \SJJ{to} \cite{Jelbart2022b,Kristiansen2019d,Miao2020,Jelbart2024} for its use in applications. 
In particular, we first consider the extended system obtained \SJ{after appending the trivial equation} $\sigma'=0$ to \SJJ{system} \eqref{eq:main0}$_{\eps=\mu\sigma}$, recall \eqref{eq:mu_scaling}. This system loses smoothness along $(r_a,r_b,p_a,p_b,\sigma)\in \Sigma\times \{0\}$, but we gain smoothness \SJJ{after applying} a sequence of transformations that blow up $\Sigma \times \{0\}$; the result is illustrated schematically in Figure \ref{fig:blow-up0} using a projection \SJJ{onto} $(p_a,p_b,\sigma)$-space. For the specific form of the blow-up transformations in the context of \SJ{system} \eqref{eq:main0}, we draw inspiration from the MSc thesis \cite{Frieder2018}, where the planar ``fundamental problem" from \cite{Plahte2005} was analysed.

\begin{figure}[t!]
\centering
\subfigure[\qquad \qquad \ \ ]{\includegraphics[scale=0.35]{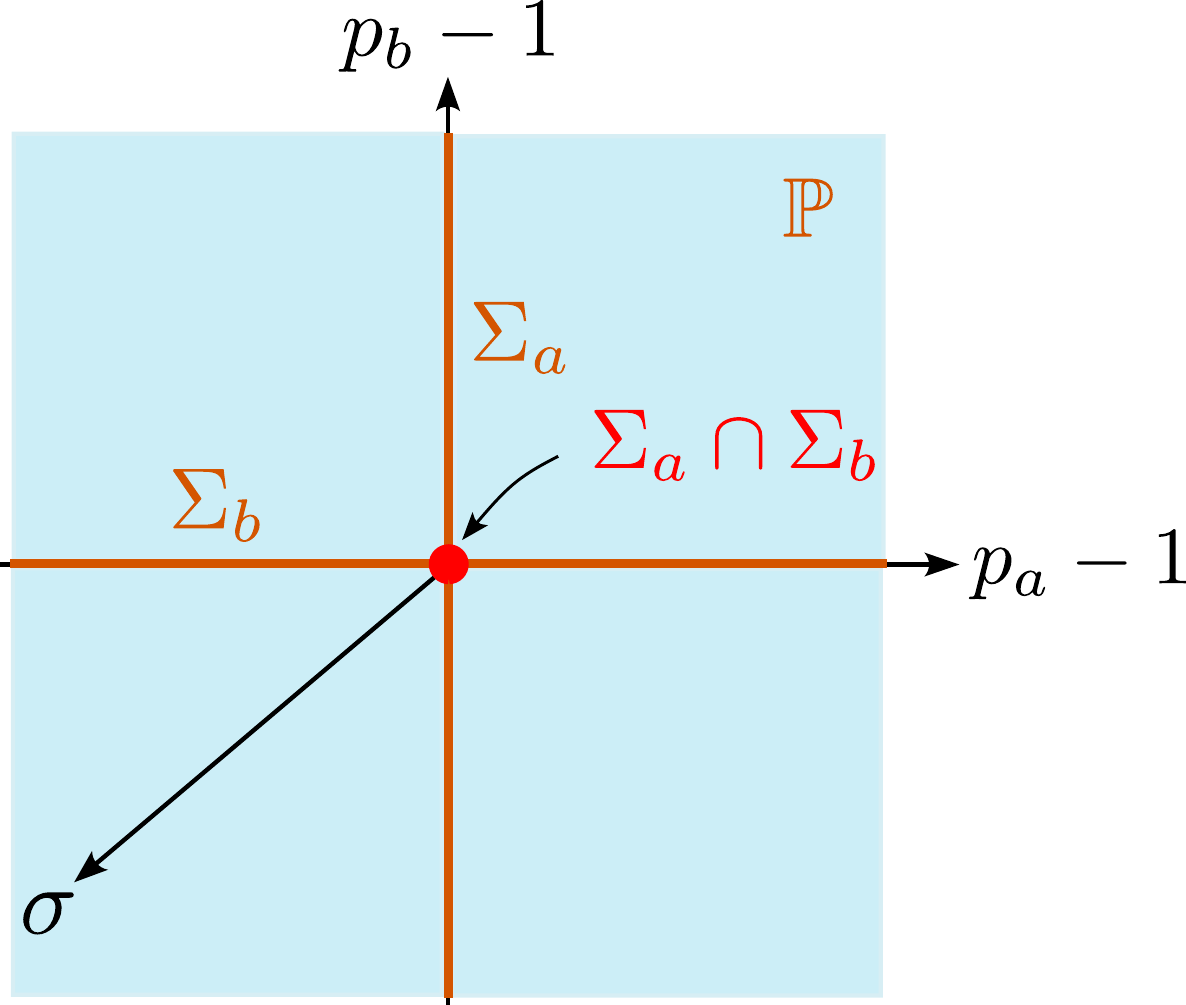}} \ \ 
\subfigure[\qquad \qquad]{\includegraphics[scale=0.35]{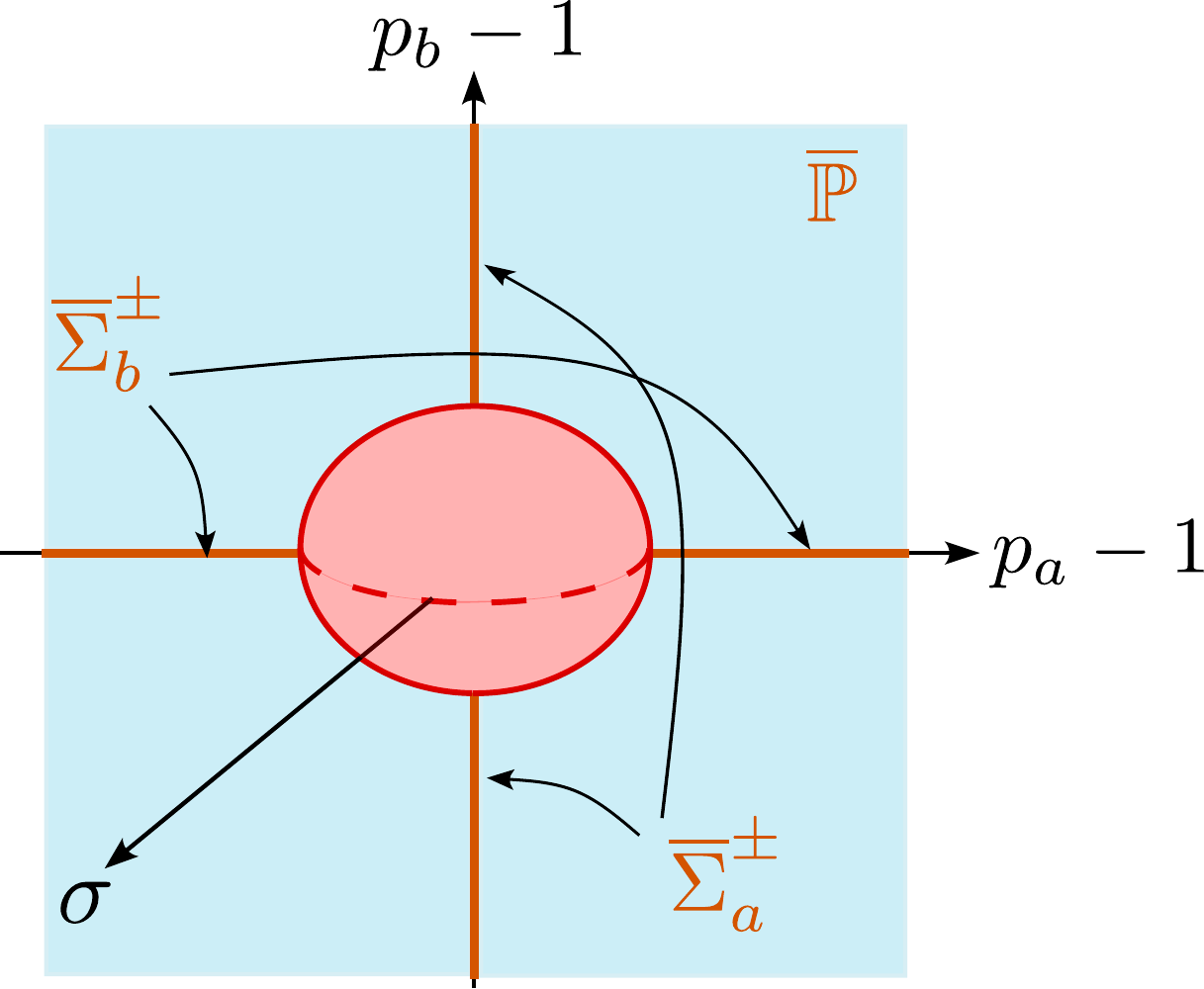}}
\quad \quad \subfigure[\qquad \qquad \ ]{\includegraphics[scale=0.35]{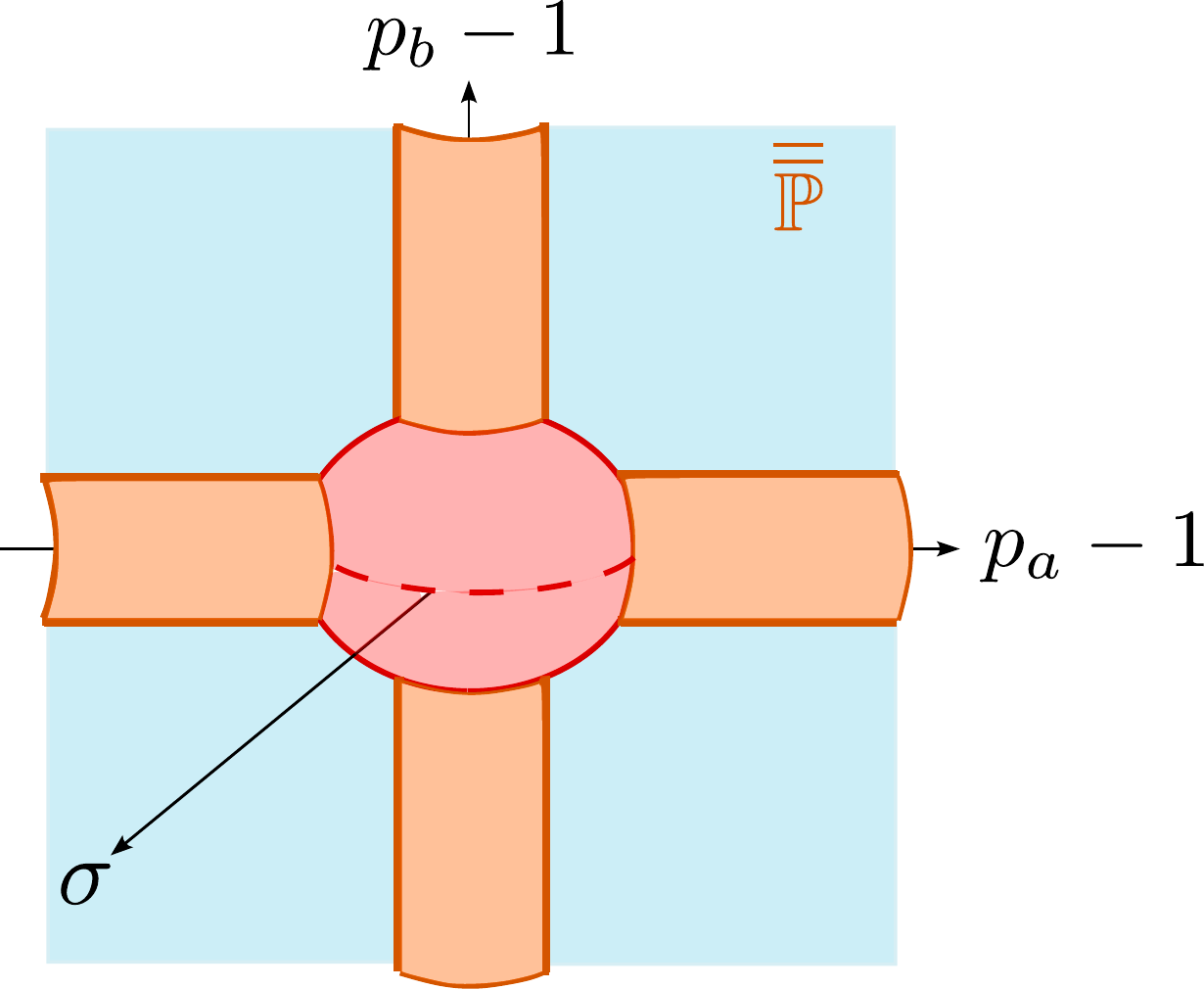}}
\caption{\SJJ{Schematic representation of the sequence of blow-up transformations which allow us to resolve the loss of smoothness along the singular set $\Sigma_a\cup \Sigma_b$ in system \eqref{eq:main0} when $\sigma = 0$}. In (a), $\Sigma_a$ and $\Sigma_b$ are illustrated (as projections) within the $(p_a,p_b)$-plane. To resolve the lack of smoothness, we first resolve the intersection $\Sigma_a\cap \Sigma_b$ (red) \SJJ{with a} spherical blow-up transformation. This gives the intermediate blown\SJ{-}up space $\overline{\mathbb P}$ in (b). Subsequently, we blow-up the remaining singularities $\overline \Sigma_b^\pm$ and $\Sigma_a^\pm$ along the axes, using cylindrical blow-ups. This gives the final blown up space $\overline{\overline{\mathbb P}}$ in (c).}
\label{fig:blow-up0}
\end{figure}

\SJJ{We} will show that \SJJ{the QSSR constraint in} \eqref{eq:M0} defines a normally hyperbolic and attracting critical manifold $\overline{\overline{ \mathcal M}}$ for ${\mu}=0$ on the blown up space \SJJJ{obtained after consecutive blow-ups}. We emphasise that this fact relies upon the scaling \eqref{eq:mu_scaling}. In this way, we obtain a locally invariant slow manifold $\overline{\overline{\mathcal M}}_{\mu}$ for all $0<\mu<\mu_0$ by Fenichel's coordinate-independent theory \SJJ{for} normally hyperbolic invariant manifolds (see \cite[Thm.~9.1]{Fenichel1979} in particular), and finally
$\mathcal M_{\mu}(\sigma)$ in Theorem~\ref{thm:thm0} upon the process of `blowing down' to the original $(r_a,r_b,p_a,p_b)$-space. As stated, the perturbed manifold $\mathcal M_\mu(\sigma)$ is $C^0$ $\mathcal O(\mu)$-close to \SJJ{the QSSR manifold defined via \eqref{eq:M0},} uniformly with respect to $\sigma>0$. Higher order $C^k$-closeness is not possible (globally) as the derivatives become unbounded with respect to $\sigma\to 0$. This should be clear enough from the dependency of $\sigma^{-1}$ in the equations for $r_a$ and $r_b$ in \SJJ{system} \eqref{eq:main0}. However, these unbounded derivatives are `tamed' in our blow-up coordinates. At the same time, the blow up as $\sigma\to 0$ is only relevant near $\Sigma$; away from these sets we have $C^k$-closeness (in agreement with the statement in Remark \ref{rem:nonexistence}).


\SJJ{Once we have proven Theorem \ref{thm:thm0}, we show how the results} can be used \SJJJ{for} further analytical and numerical investigations of the dynamics of system \eqref{eq:main0}. Although we focus entirely on the dynamics of a simple, $4$-dimensional GRN model, our results suggest that the dynamics of high-dimensional GRN systems in the general form \eqref{eq:grn_network} may depend in an important way on the relative size of $\eps$ and the small parameters $1/n_i$ associated with switching effects. The question of whether and how our results can be extended to the study of the generally high-dimensional network \eqref{eq:grn_network} is the topic of on-going work. \SJJJ{We are hopeful that a geometric blow-up based approach will provide a means to study the effects of} 
both (i) and (ii) in a single, geometrically informative and rigorous mathematical framework.

\subsection{Overview}
The \SJJJ{remainder of the paper} is structured as follows: 
\rsp{In Section~\ref{sec:proofs} we prove Theorem~\ref{thm:thm0}; the proof of Assertion \ref{existence} is deferred to Appendix~\ref{app:blow-up} for brevity.}
\SJJ{In Section~\ref{sec:the_model} we use Theorem \ref{thm:thm0} \SJJJ{for} further analytical and numerical investigations \SJJJ{of} the dynamics of system \eqref{eq:main0} by using numerical computations and our reduction in Theorem~\ref{thm:thm0}, thereby demonstrating the applicability of our results}. We conclude with a summary and outlook in Section \ref{sec:summary_and_outlook}.

\section{Proofs} \label{sec:proofs}
\rsp{The proof of Assertion \ref{existence} of Theorem \ref{thm:thm0} is available in Appendix~\ref{app:blow-up}, see also \ref{app:existence}}. We prove the \rsp{remaining assertions} in the following subsections.



\subsection{Assertion \ref{kappa}}
\label{sec:kappa}

In this section, we work \rsp{with following system:
\begin{equation}\label{eq:K2_system_2}
\begin{aligned}
r_a' &= {\phi}(v_2) - r_a , \\
r_b' &= 1-{\phi}(u_2) - \gamma r_b , \\
u_2' &= \mu G_a(r_a, \sigma u_2) , \\
v_2' &= \mu  G_b(r_b, \sigma v_2) ,
\end{aligned}
\end{equation}
see also \eqref{eq:K2_system} (with $\eta_2=\sigma$) in Appendix~\ref{app:blow-up}, which is obtained from system \eqref{eq:main0} using $\eps = \mu \sigma$ as in \eqref{eq:mu_scaling}, together with the coordinate transformation
\begin{equation}
\label{eq:K2_coords}
(p_a,p_b)=(\me^{\sigma u_2},\me^{\sigma v_2}),
\end{equation}
and the notation
\begin{equation}\nonumber
\begin{aligned}
G_a(x,y) := \xi_a x \me^{-y} - 1 , \qquad 
G_b(x,y) := \delta( \xi_b x \me^{-y} - 1 ).
\end{aligned}
\end{equation}
Due to the form of the coordinate transformation \eqref{eq:K2_coords}, system \eqref{eq:K2_system_2} describes the dynamics in a `scaling chart' which is centered about the intersection point $\{(1,1)\} = \Sigma_a \cap \Sigma_b$ (since $p_a \sim 1 + \mathcal \sigma u_2$ and $p_b \sim 1 + \sigma v_2$ as $\sigma \to 0$).} We consider $\mu$ and $\sigma$ as independent small parameters. We are interested in the dynamics on \rsp{a} two-dimensional \rsp{attracting} slow manifold $\mathcal M_{\mu,2}=\mathcal M_{\mu,2}(\sigma)$ of \eqref{eq:K2_system_2}. \rsp{The existence and properties of $\mathcal M_{\mu,2}=\mathcal M_{\mu,2}(\sigma)$ given in Lemma \ref{lem:slow_manifolds_2} of Appendix \ref{app:blow-up}}. Restricting \eqref{eq:K2_system_2} to $\mathcal M_{\mu,2}(\sigma)$ \rsp{(using the form in Lemma \ref{lem:slow_manifolds_2})}, we obtain the planar system
\begin{equation}
\label{eq:K2_system_slow_manifold}
\begin{split}
\dot u_2 &= G_a ({\phi}(v_2), \sigma u_2) + \mu \xi_a \me^{- \sigma u_2} \Omega_a(u_2, v_2, \sigma) + \mathcal O(\mu^2) , \\
\dot v_2 &= \delta \left( G_b (\gamma^{-1} (1-{\phi}(u_2)), \sigma v_2) + \mu \xi_b \me^{- \sigma v_2} \Omega_b(u_2, v_2, \sigma) + \mathcal O(\mu^2) \right) ,
\end{split}
\end{equation}
where the $\mathcal O(\mu^2)$ terms are uniform with respect to $\sigma$. In particular, since \eqref{eq:K2_system_2} depends regularly upon $\sigma$, these $C^k$-smooth remainder terms can in this chart be differentiated with respect to $u_2,v_2$ and $\sigma$ without changing the order. 

The limiting/reduced flow on $\mathcal M_{2,\mu}(\sigma)$ when $\mu = \sigma = 0$ is governed by the system
\begin{equation}
\label{eq:K2_reduced_mu}
\begin{split}
\dot u_2 &= \xi_a {\phi}(v_2) - 1 , \\
\dot v_2 &= \delta \left( \xi_b \gamma^{-1} (1-{\phi}(u_2)) - 1 \right) ,
\end{split}
\end{equation}
which has a unique equilibrium point $$q_2 : (u_2^\ast, v_2^\ast) = (\ln(\xi_b - \gamma) - \ln (\gamma), -\ln(\xi_a - 1))$$ if and only if $\xi_a > 1$ and $\xi_b > \gamma$. Using this fact and applying the implicit function theorem, it follows that for each fixed choice of $\xi_a > 1$, $\xi_b > \gamma$, system \eqref{eq:K2_system} has a unique equilibrium 
\begin{align}\label{eq:q2musigma}
q_{2,\mu,\sigma} = q_2 + \mathcal O(\mu,\sigma),
\end{align} for all $\mu, \sigma > 0$ sufficiently small. The Jacobian matrix associated with the reduced system \eqref{eq:K2_reduced_mu} is given by
\[
J_0(u_2,v_2) =
\begin{pmatrix}
0 & \xi_a {\phi}(v_2) (1-{\phi}(v_2)) \\
- \delta \xi_b \gamma^{-1} {\phi}(u_2) (1-{\phi}(u_2)) & 0
\end{pmatrix} .
\]
Since $\operatorname{tr} J_0(u_2,v_2) \equiv 0$ and $\det J_0(u_2, v_2) > 0 $ for all $(u_2, v_2) \in \R^2$ and $(\xi_a, \xi_b) \in (1,\infty) \times (\gamma, \infty)$, it follows that $q_2$ is a center. 
\begin{remark}
\label{rem:Hamiltonian}
\SJ{Although \rspp{our arguments do not rely on it explicitly}, we note that the} limiting system in \eqref{eq:K2_reduced_mu} is Hamiltonian with Hamiltonian function
\[
H(u_2, v_2) := \frac{\delta \xi_b}{\gamma} \ln \left( 1 + \me^{u_2} \right) + \xi_a \ln \left( 1 + \me^{v_2} \right) - \delta \left( \frac{\xi_b}{\gamma} - 1 \right) u_2 - v_2 .
\]
\end{remark}
\SJJ{In} order to determine the stability of $q_{2,\mu,\sigma}$, we need to look to the next order in $\mu, \sigma$. Let $J_{\mu,\sigma}(u_2,v_2)$ denote the Jacobian matrix associated with system \eqref{eq:K2_system_slow_manifold}. Direct calculations together with our calculations for $\mu = \sigma = 0$ show that 
\[
\begin{split}
\operatorname{tr} J_{\mu,\sigma}&(u_2, v_2) =
- \sigma \left[ \xi_a {\phi}(v_2) \me^{-\sigma u_2} + \delta \xi_b \gamma^{-1} (1-{\phi}(u_2)) \me^{- \sigma v_2} \right] \\
&+ \mu \left[ \xi_a \frac{\partial}{\partial u_2} \left( \me^{- \sigma u_2} \Omega_a(u_2, v_2, \sigma) \right) + \delta \xi_b \frac{\partial}{\partial v_2} \left( \me^{- \sigma v_2} \Omega_b(u_2, v_2, \sigma) \right) \right] + \mathcal O(\mu^2) .
\end{split}
\]
Expanding about $\sigma = 0$, using that
\[
\begin{split}
\frac{\partial \Omega_a}{\partial u_2} &= \delta \xi_b \gamma^{-1} {\phi}(u_2) (1-{\phi}(u_2)) {\phi}(v_2) h^-(v_2) + \mathcal O(\sigma) , \\
\frac{\partial \Omega_b}{\partial v_2} &= \gamma^{-2} \xi_a {\phi}(u_2) (1-{\phi}(u_2)) {\phi}(v_2)(1-{\phi}(v_2)) + \mathcal O(\sigma) ,
\end{split}
\]
which follows from \eqref{eq:omegaab},
and evaluating the expression at $q_{{\mu,\sigma,2}}$, leads to the asymptotic formula
\begin{equation}
\label{eq:trace}
\operatorname{tr} J_{\mu,\sigma}(q_{{\mu,\sigma,2}}) =  - \sigma (1 + \delta) + \mu \frac{ (\xi_a - 1) (\xi_b - \gamma )}{\xi_a \xi_b} \frac{\delta ( 1 + \gamma)}{\gamma} + \mathcal O( |(\mu, \sigma)|^2 ) .
\end{equation}


\begin{lemma}\label{lem:Hopf}
Suppose $(\xi_a,\xi_b) \in \SJJ{\Lambda}$, \SJJ{for some compact and connected set} $\SJJ{\Lambda} \subset (1,\infty)\times (\gamma,\infty) $. \SJJ{There exists a $\sigma_0 > 0$ and a $C^k$-smooth function $\mu_{\textrm{Hopf}} : \Lambda \times [0, \sigma_0) \to \R$ such that system}
\eqref{eq:K2_system_2} undergoes Hopf bifurcations if and only if 
\begin{align*}
\mu = \mu_{\textrm{Hopf}}(\xi_a,\xi_b,\sigma).
\end{align*}
Here $\mu_{\textrm{Hopf}}(\xi_a,\xi_b,0)=0$ and 
\begin{align}\label{eq:dmuH}
\frac{\partial \mu_{\textrm{Hopf}}}{\partial \sigma}(\xi_a,\xi_b,0) = \frac{\xi_a \xi_b}{(\xi_a-1)(\xi_b-\gamma)} \frac{\gamma (1+\delta)}{\delta(1+\gamma)}.
\end{align}

\end{lemma}
\begin{proof}
Hopf bifurcations are characterized by $\operatorname{tr} J_{\mu,\sigma}(q_{{\mu,\sigma,2}}) =0$ and $\det J_{\mu,\sigma}(q_{{\mu,\sigma,2}}) > 0$. \SJJ{Since} $\det J_0(q_{{2}})\ne 0$ and
\begin{align*}
\operatorname{tr} J_{0}(q_{{2}})=0,\quad \frac{\partial \operatorname{tr} J_{\mu,\sigma}(q_{{\mu,\sigma,2}})}{\partial \mu}\bigg|_{\mu=\sigma=0} = \frac{ (\xi_a - 1) (\xi_b - \gamma )}{\xi_a \xi_b} \frac{\delta ( 1 + \gamma)}{\gamma} \ne 0,
\end{align*}
we obtain \SJJ{the function $\mu_{\textrm{Hopf}}$ by} solving 
$\operatorname{tr} J_{\mu,\sigma}(q_{{\mu,\sigma,2}}) =0$ locally for $\mu$ as a smooth function of $(\xi_a,\xi_b)\in \SJJ{\Lambda}$ and $0 \le \sigma<\sigma_0$, \SJJ{for} $\sigma_0>0$ small enough, using the implicit function theorem.
\end{proof}
\begin{remark}
It is easy to show that the Hopf bifurcation in Lemma \ref{lem:Hopf} is unfolded with respect to $(\xi_a,\xi_b)$, but we have not been able to analyze the Lyapunov coefficient (see e.g. \cite[p.~211]{Chow1994}) analytically. The expressions are overwhelming. However, numerical investigations in MatCont \cite{MATCONT} suggest that the 
bifurcation is always supercritical.
\end{remark} 

\medskip

\SJJ{We are now in a position to state and prove a more precise version of Theorem \ref{thm:thm0} assertion \ref{kappa}.} 
\begin{proposition}\label{prop:kappa}
\SJ{Recall the definition of $\alpha > 0$ in \eqref{eq:alpha} and consider system} \eqref{eq:K2_system_2} along rays in the $(\sigma,\mu)$-plane defined by $\mu=c\sigma$, $\sigma>0$, with $c>0$, $c\ne \alpha$, fixed. Then there are only Hopf bifurcations (for $0<\sigma\ll 1$) along such rays with $c> \alpha$. More precisely, we have the following:
\begin{enumerate}
\item \label{ass1} Suppose that $c>\alpha$. Then there is a $\sigma_0>0$ small enough such that for all $0<\sigma<\sigma_0$ the system \eqref{eq:K2_system_2} with $\mu=c\sigma$ has a Hopf bifurcation 
\KUK{for some $(\xi_a,\xi_b) \in (1,\infty) \times (\gamma,\infty)$}.
\item Suppose that $0<c<\alpha$ and fix any compact domain $\SJJ{\Lambda} \subset (1,\infty)\times (\gamma,\infty)$. Then there is a $\sigma_0>0$ small enough such that for all $0<\sigma<\sigma_0$ the system \eqref{eq:K2_system_2} with $\mu=c\sigma$ and bifurcation parameters  $(\xi_a,\xi_b)\in \SJJ{\Lambda}$ does not support a Hopf bifurcation.
\end{enumerate}

\end{proposition}
\begin{proof}
Given that $\mu=c\sigma$, we conclude by Lemma \ref{lem:Hopf} that we have a  Hopf bifurcation if and only if
\begin{align*}
\mu = c\sigma = \mu_{\textrm{Hopf}}(\xi_a,\xi_b,\sigma) = \frac{\partial \mu_{\textrm{Hopf}}}{\partial \sigma} (\xi_a,\xi_b,0)\sigma+\mathcal O(\sigma^2),
\end{align*}
or 
\begin{align}\label{eq:cond1}
c=\frac{\partial \mu_{\textrm{Hopf}}}{\partial \sigma} (\xi_a,\xi_b,0)+\mathcal O(\sigma).
\end{align}
The main insight is then that the function 
$$(\xi_a,\xi_b)\mapsto \frac{\partial \mu_{\textrm{Hopf}}}{\partial \sigma}(\xi_a,\xi_b,0), \quad \xi_a>1,\,\xi_b>\gamma,$$ is \SJ{surjective on} the interval $(\alpha,\infty)$. This follows directly from \eqref{eq:dmuH}. 
Consequently, if $c>\alpha$ then we can solve \eqref{eq:cond1} for some $(\xi_a,\xi_b)$ for all $0<\sigma\ll 1$ and there is a Hopf bifurcation. In contrast, if $c<\alpha$ and we fix $\SJJ{\Lambda}$, then the right hand side of \eqref{eq:cond1} is strictly less than $\alpha$ for all $0<\sigma\ll 1$ and all $(\xi_a,\xi_b)\in \SJJ{\Lambda}$, \SJJ{in which case} there is no Hopf bifurcation. 
\end{proof}
\begin{remark}
\KUK{In assertion \ref{ass1} where $c>\alpha$, there are Hopf bifurcations along a curve (a hyperbola for any $0<\sigma\ll 1$, see Figure \ref{fig:eq_path_Hopf} (b) below) in the $(\xi_a,\xi_b)$-plane. To see this, notice that for each $\xi_b>\gamma$ ($\xi_a>1$), we can solve \eqref{eq:cond1} for $\xi_a>1$ ($\xi_b>\gamma$, respectively) for all $0<\sigma\ll 1$.} 
\end{remark}


\subsection{Assertion \ref{nonexistence}}
In this section, we suppose that $\mu>0$ is fixed as $\sigma \to 0$. We then consider \SJJJ{system} \eqref{eq:main0} with $\eps=\mu \sigma$ in the following form:
\begin{equation}\label{eq:fuck}
\begin{split}
r_a' &={\phi}(\sigma^{-1}\log p_b)-r_a,\\
r_b' &=1-\phi(\sigma^{-1}\log p_a)-\gamma r_b,\\
p_a'&=\mu \sigma (\xi_a r_a-p_a),\\
p_b'&=\mu \sigma \delta( \xi_b r_b-p_b).
\end{split}
\end{equation}
Now, $\mathcal M(\sigma)$ is given by 
\begin{align}\label{eq:QSSR32}
r_a={\phi}(\sigma^{-1}\log p_b),\qquad 
r_b=\gamma^{-1} (1-\phi(\sigma^{-1}\log p_a)).
\end{align}
We then consider the Lie-derivative $\mathcal L$ of the function
$$(r_a,r_b,p_a,p_b)\mapsto r_a-{\phi}(\sigma^{-1}\log p_b),$$ in the direction of the vector-field induced by \eqref{eq:fuck} at a point on $\mathcal M(\sigma)$ with $p_a>1$ and $p_b=1$. This gives 
\begin{align*}
\mathcal L=-\phi'(0)\mu \delta ( \xi_b \gamma^{-1} (1-\phi(\sigma^{-1}\log p_a))-1) = \frac14 \mu \delta(1+\mathcal O(\sigma)) .
\end{align*}
Since $\mathcal L > c > 0$ for a positive constant $c$ for all $0<\sigma \ll 1$, there can be no invariant manifold $o(1)$-close to $\mathcal M(\sigma)$ as $\sigma\to 0$ in a uniform neighbourhood of this point. 
The result therefore follows upon letting $K'$ be so that $(p_a,1)\in K'$ for some $p_a>1$. (Of course, we could also have arrived at the same result by working along $p_a=1$.)
\qed

\section{Dynamics}\label{sec:the_model}

\SJ{In this section, we combine the results in Theorem \ref{thm:thm0} with analytical and numerical investigations into} the dynamics of \SJ{system} \eqref{eq:main0}. We start with the following result. 

\begin{lemma}
\label{lem:equilibrium}
System \eqref{eq:main0} with $\sigma,\eps>0$ has a unique equilibrium $q : (r_a^\ast, r_b^\ast, \xi_a r_a^\ast, \xi_b r_b^\ast)$, where the coordinates $(r_a^\ast, r_b^\ast) \in (0,1 / \xi_a) \times (1 / \gamma \xi_b)$ are uniquely determined by the equations
\begin{equation}
\label{eq:nullclines}
r_a = h^+\left( \xi_b r_b ; 1, \sigma^{-1} \right) ,
\qquad 
r_b = \gamma^{-1} h^-\left( \xi_a r_a ; 1, \sigma^{-1} \right) .
\end{equation}
The location of $q$ depends on $\xi_a, \xi_b$ and $\sigma$, but not on $\eps$.
\end{lemma}

\begin{proof}
Equilibria of system \eqref{eq:main0} with $\eps > 0$ have to satisfy $p_a = \xi_a r_a$ and $p_b = \xi_b r_b$. Substituting these expressions into $r_a' = 0$ and $r_b' = 0$ yields the conditions in \eqref{eq:nullclines}. The existence of a unique equilibrium point
$q : (r_a^\ast, r_b^\ast, \xi_a r_a^\ast, \xi_b r_b^\ast)$ 
follows from the monotonicity properties of the Hill functions $h^\pm$, see Figure \ref{fig:Hill_functions}. Indeed, define
\begin{equation}
\label{eq:eq_cond}
F(r_a) := r_a - h^+( \xi_b \gamma^{-1}  h^-(\xi_a r_a ; 1,\sigma^{-1}) ; 1,\sigma^{-1} ).
\end{equation}
It is clear that \eqref{eq:nullclines} is equivalent \SJJ{to} the equations $$F(r_a)=0,\qquad r_b =\gamma^{-1} h^-\left( \xi_a r_a ; 1, \sigma^{-1} \right).$$ Then since $$F(0) = - h^+( \xi_b \gamma^{-1}  \tfrac12 ; 1,\sigma^{-1} )< 0, \quad \lim_{r_a \to \infty} F(r_a) = \infty\quad \mbox{and}\quad (\partial F / \partial r_a) (r_a) > 0,$$ for all $\sigma>0$, the intermediate value theorem implies the existence of a unique solution $r_a = r_a^\ast > 0, r_b=r_b^\ast$. 
Finally, the fact that $(r_a^\ast, r_b^\ast) \in (0, 1 / \xi_a) \times (0, 1/ \gamma \xi_b)$ follows from the expressions in \eqref{eq:nullclines} and the fact that range of the Hill functions $h^\pm$ is $(0,1)$.
\end{proof}

Consider the $(p_a,p_b)$-plane and let $I,\ldots,IV$ denote the four quadrants that are obtained from the partition induced by $\widetilde \Sigma = \widetilde \Sigma_a \cup \widetilde \Sigma_b$, where
\begin{equation}
\label{eq:switching_manifold_reduced}
\widetilde \Sigma_a := \left\{ (1,p_b) : p_b \geq 0 \right\} , \quad
\widetilde \Sigma_b := \left\{ (p_a,1) : p_a \geq 0 \right\} ,
\end{equation} 
see also Figure \ref{fig:eq_path} (a). \SJJ{The fact that QSSR is valid uniformly in $\sigma$ in regions of the phase space that are bounded away from $\Sigma$ (recall Remark \ref{rem:nonexistence}), motivates us to consider the QSSR system} \eqref{eq:M0red} in the singular limit $\sigma\to 0$. \SJJ{This is a} 
piecewise-linear \SJ{(PWL)} system \SJJ{of the form}
\begin{equation}
\label{eq:pwl_limit_reduced}
\begin{split}
I :&
\begin{cases}
\dot p_a &= - p_a , \\
\dot p_b &= \delta \left( \frac{\xi_b}{\gamma} - p_b \right) ,
\end{cases}
\qquad \quad
II :
\begin{cases}
\dot p_a &= - p_a , \\
\dot p_b &= - \delta p_b ,
\end{cases}
\\
III :&
\begin{cases}
\dot p_a &= \xi_a - p_a , \\
\dot p_b &= - \delta p_b ,
\end{cases}
\qquad \qquad \quad
IV :
\begin{cases}
\dot p_a &= \xi_a - p_a , \\
\dot p_b &= \delta \left( \frac{\xi_b}{\gamma} - p_b \right) ,
\end{cases}
\end{split}
\end{equation}
where the overdot denotes differentiation with respect to slow time $\tau = \eps t$ and $\widetilde \Sigma$ is the switching manifold. 
As all four systems in \eqref{eq:pwl_limit_reduced} are linear and decoupled, they can easily be solved directly. We omit the calculations here because they will not be necessary for our purposes, however, we summarise a number of important qualitative features of the dynamics in the following result. The proof can be found in Appendix~\ref{app:prop} below.
%

\begin{proposition}
\label{prop:PWS_dynamics}
Consider the PWL protein-only system defined by the systems in \eqref{eq:pwl_limit_reduced} on $\R^2 \setminus \widetilde \Sigma$. The following assertions are true:
\begin{enumerate}
\item[(i)] For all $(\xi_a, \xi_b) \in \R \cup [0,1)$ there is a unique equilibrium $q_{III} : (0,\xi_b / \gamma) \in III$, which is a stable node.
\item[(ii)] For all $(\xi_a, \xi_b) \in [0,1) \cup (\gamma,\infty)$ there is a unique equilibrium $q_{II} : (\xi_a,\xi_b / \gamma) \in II$, which is a stable node.
\item[(iii)] There are no equilibria on $\R^2 \setminus \widetilde \Sigma$ when $(\xi_a, \xi_b) \in (1,\infty) \cup (\gamma,\infty)$.
\item[(iv)] For fixed $\xi_a \in [0,1)$, 
\begin{align*}
\lim_{\xi_b\to \gamma^-} q_{III} = (0,1)\in \widetilde \Sigma_b.
\end{align*}
\item[(v)] For fixed $\xi_a \in [0,1)$, 
\begin{align*}
\lim_{\xi_b\to \gamma^+} q_{II} = (\xi_a,1)\in \widetilde \Sigma_b.
\end{align*}
\item[(vi)] For fixed $\xi_b \in [0,\gamma)$, 
\begin{align*}
\lim_{\xi_a\to 1^-} q_{II} = (1,\gamma^{-1} \xi_b) \in \widetilde \Sigma_a.
\end{align*}
\item[(vii)] For fixed $\xi_a > 1$,
\begin{align*}
\lim_{\xi_b\to \gamma^-} q_{III} = (0,1)\in \widetilde \Sigma_b .
\end{align*}

\item[(viii)] Suppose that $\xi_a>1$ and $\xi_b>\gamma$ and consider $(p_a,p_b)(0)\ne (1,1)$. Then $(p_a,p_b)(\SJJ{\tau})\rightarrow (1,1)$ for $\SJJ{\tau} \rightarrow T_{\rm{max}}$, where $T_{\rm{max}}\in (0,\infty]$ is the maximal time of existence. 
\end{enumerate}
\end{proposition}

\begin{figure}[t!]
\centering
\subfigure[]{\includegraphics[scale=0.5]{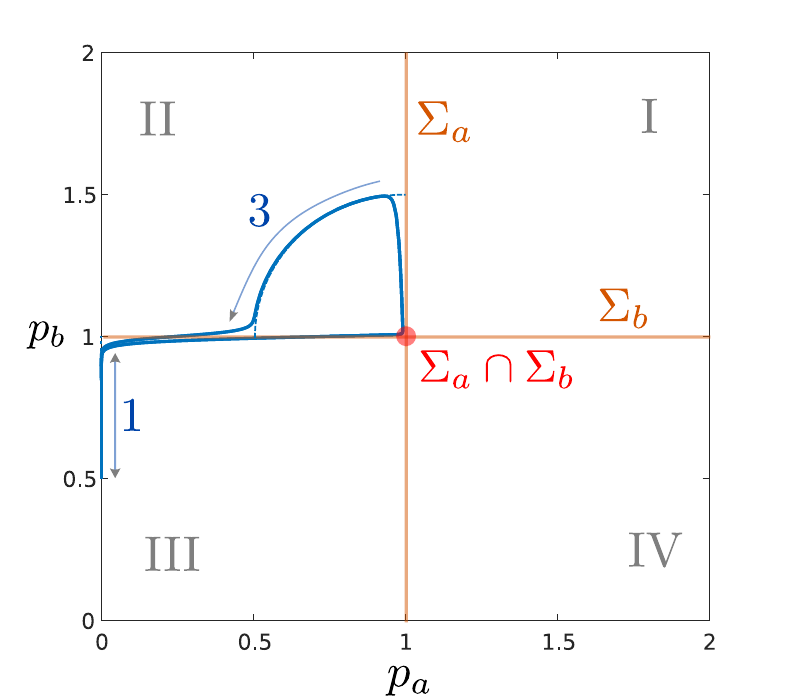} \ \ } \ \ 
\subfigure[]{\includegraphics[scale=0.5]{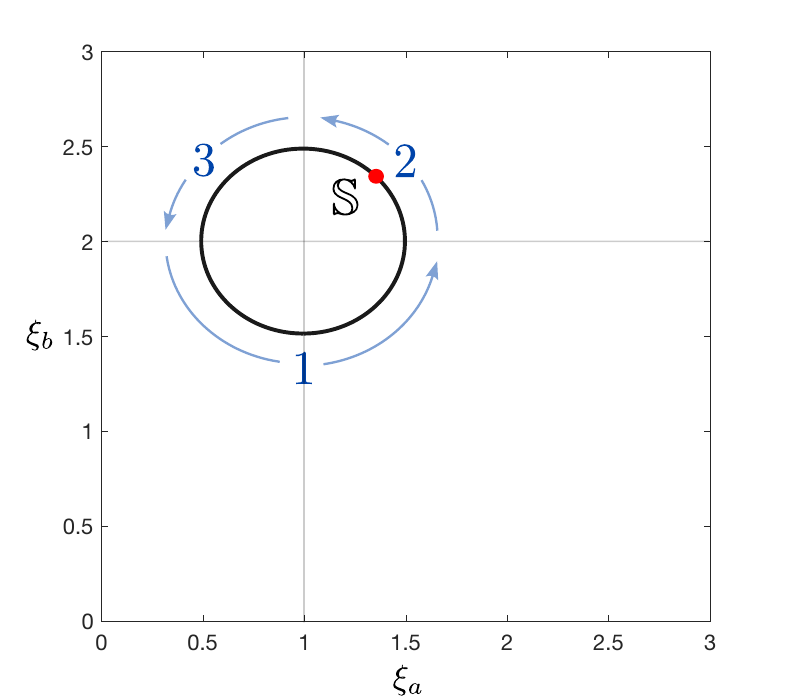} \ \ }
\caption{In (a): Projections of the switching manifolds $\Sigma_a$, $\Sigma_b$ and their intersection $\Sigma_a \cap \Sigma_b$ in the $(p_a, p_b)$-plane are shown in shaded orange and red respectively. A numerical continuation, in which the $(p_a,p_b)$-coordinates of the unique equilibrium $q$ of system \eqref{eq:main0} for parameter values $\eps = 5 \times 10^{-5}$, $\sigma = 10^{-2}$, $\gamma = 2$, $\delta = 3$, and $(\xi_a,\xi_b)$-values in a circle $\mathbb S$ of radius $1/2$ centered at $(1, \gamma) = (1, 2)$, is overlaid in blue. The dashed blue line in regions II and III show $q_{II}$ (barely visible along $p_a=0$) and $q_{III}$ from Proposition \ref{prop:PWS_dynamics}. In (b): The circle $\mathbb S$ in $(\xi_a,\xi_b)$-space which corresponds to the set of points on the closed blue loop in the left panel figure is shown in black. The labels ``$1$'' and ``$3$'' correspond to the parts of the loop in the left panel figure with the corresponding label. The label ``$2$" does not appear in (a) because it is `hidden' near $\Sigma_a\cap \Sigma_b$ (it is described in the scaled coordinates of $K_2$). \rsp{The red dot on $\mathbb S$ in (b) indicates the parameter values used in Figure \ref{fig:lcnolc}.}}
\label{fig:eq_path}
\end{figure}

\begin{remark}
The phenomenon in (iv)-(vii), where $q_{III}$ and $q_{II}$ collide with the different branches of $\Sigma$, is known as boundary-equilibrium bifurcation in PWS systems. We refer to \cite{Filippov1960,Hogan2016,Kuznetsov2003} for definitions and classifications of `generic' boundary-equilibria bifurcations in planar Filippov systems, and to \cite{Jelbart2021,Jelbart2020d} for detailed analyses and a classification of `generic' singularly perturbed boundary equilibria which occur in smooth planar systems that limit to a PWS system \SJJ{(as is the case for system \eqref{eq:main0} when $\sigma \to 0$)}. We also refer to \cite{Kristiansen2019d}, where a singularly perturbed boundary-node bifurcation is shown to play an important role in describing the rapid onset of large amplitude oscillations in a model for substrate-depletion oscillation. The cases in (iv)-(vii) in Proposition \ref{prop:PWS_dynamics} are examples of degenerate boundary-node bifurcations. The degeneracy stems from the fact that $q_{III}$ and $q_{II}$ have an associated eigenvector which intersects $\widetilde \Sigma$ non-transversally at the collision.
\end{remark}
%
%

Now consider a circle $\mathbb S$ in the $(\xi_a,\xi_b)$ parameter plane of radius $< \min\{1, \gamma\}$ centered around $(1,\gamma)$, see Figure \ref{fig:eq_path} (b). We consider $\gamma=2$, $\delta=3$ and a radius equal to $\frac12$ to be specific. Then on the lower left part of $\mathbb S$, where $\xi_a<1$ and $\xi_b<\gamma=2$, $q_{III}$ of Proposition \ref{prop:PWS_dynamics} traces out a curve in the $(p_a,p_b)$-plane which is shown in dashed blue in Figure \ref{fig:eq_path} (a) (it is contained within $p_a=0$ and \SJJ{is barely visible because it aligns} with the blue curve). Similarly on the upper left part of $\mathbb S$, where $\xi_a<1$ and $\xi_b>\gamma=2$, $q_{II}$ of Proposition \ref{prop:PWS_dynamics} traces out a curve (also a quarter circle and shown in dashed blue in Figure \ref{fig:eq_path} (a)). We then perform numerical computations for $\eps = 5\times 10^{-5}$, $\sigma=10^{-2}$ (and therefore $\mu=5\times 10^{-3}$ cf.~\eqref{eq:mu_scaling}), to follow the unique equilibrium $q$ of \SJJ{system} \eqref{eq:main0} as $(\xi_a,\xi_b)$ transverses around $\mathbb S$. We describe the results as \SJJ{below}.

We start with a $(\xi_a, \xi_b)$ value within the lower left quadrant of $\mathbb S$, in which case assertion (i) of Proposition \ref{prop:PWS_dynamics} implies the existence of a stable node $q_{III}$. As we begin to move counterclockwise around $\mathbb S$, $q$ initially moves down the vertical axis until it reaches a minimum for parameter values at the `south pole' of $\mathbb S$, after which it begins to move up the $p_b$-axis. This accounts for part `1' of the loop depicted in Figure \ref{fig:eq_path} (b) and is in agreement with the corresponding path of $q_{III}$ (in dashed blue) of Proposition \ref{prop:PWS_dynamics}. The degenerate boundary-node bifurcation described in assertion (iv) of Proposition \ref{prop:PWS_dynamics} occurs as $\xi_b / \gamma$ approaches $1$ from below and $q$ collides with $\widetilde \Sigma_b$, i.e.~as we approach the right-most point of $\mathbb S$ from below. Due to assertion (iii) of Proposition \ref{prop:PWS_dynamics}, there are no equilibria of \SJJ{system} \eqref{eq:pwl_limit_reduced} on $\R^2 \setminus \widetilde \Sigma$ as we continue along the top right quadrant of $\mathbb S$; this is part `2' of the loop. As we continue into the top left quadrant of $\mathbb S$, the equilibrium $q \in IV$ `emerges' from $\widetilde \Sigma_a$ via the degenerate boundary-node bifurcation described in assertion (vi) of  Proposition \ref{prop:PWS_dynamics}. The equilibrium $q$ subsequently follows a path determined by $q_{II}$ (in dashed blue) until it collides with $\widetilde \Sigma_b$ from above as the left-most point of $\mathbb S$ is approached from above (corresponding to the degenerate boundary-node described in assertion (v)); this is part `3' of the loop. Finally, in accordance with assertion (vii), $q$ emerges from $\widetilde \Sigma_b$ via another degenerate boundary-node bifurcation, and moves into $III$ as we move back into the lower left quadrant of $\mathbb S$, thereby completing the loop. Numerically we find that the equilibrium remains stable throughout (and the global attractor).

These results are consistent with Theorem \ref{thm:thm0}. In fact, Fenichel's theory guarantees that the qualitative dynamics of the original $4$-dimensional model \eqref{eq:main0} for all $0<\eps\ll 1$, $0<\sigma\ll 1$, are captured by \eqref{eq:pwl_limit_reduced} and Proposition \ref{prop:PWS_dynamics} as long as we stay uniformly away from $\Sigma$; notice, cf. Remark \ref{rem:nonexistence}, that this does not rely on \eqref{eq:mu_scaling}. 
%
However, when $\xi_a > 1$ and $\xi_b > \gamma$, assertion (iii) in Proposition \ref{prop:PWS_dynamics} states that the limiting PWL system \eqref{eq:pwl_limit_reduced} does not have an equilibrium on $\R_+^2 \setminus \widetilde \Sigma$. To track $q$ within this parameter regime we can work on the local version of $\mathcal M_\mu(\sigma)$ in the chart $K_2$, recall \eqref{eq:K2}, Section \ref{sec:kappa} and \eqref{eq:q2musigma} specifically. The coordinates in $K_2$ provide a zoom near $\Sigma_a\cap \Sigma_b$; this is consistent with Figure \ref{fig:eq_path} (part `2' of $\mathbb S$) where $q$ remains close to $\Sigma_a\cap \Sigma_b$. We also expect that it is possible to track $q$ on the remaining parts of $\mathbb S$, near $\xi_a\approx 1$ and $\xi_b\approx \gamma$, by working on the local versions of $\mathcal M_\mu(\sigma)$ in the charts $K_{ij}$, $i\in \{1,3,4,5\}$. This will involve further parameter blow-up, see e.g.~\cite{Baumgartner2024,Jelbart2020d,Kristiansen2019d}.


In Figure \ref{fig:eq_path_Hopf}, we repeat the numerical computations with the same data that was used to generate Figure \ref{fig:eq_path} except that $\eps$ has increased from $5 \times 10^{-5}$ to $5 \times 10^{-3}$ (so that $\mu=\frac12$ in Figure \ref{fig:eq_path_Hopf}, cf. \eqref{eq:mu_scaling}). Following this increase in $\eps$, we now identify two supercritical Hopf bifurcations using MatCont \cite{MATCONT} as we traverse the same circle in the $(\xi_a, \xi_b)$-plane. Figure \ref{fig:eq_path_Hopf} (b) shows the output of a $2$-parameter continuation in MatCont \cite{MATCONT} of the Hopf curve (note that the two bifurcations points in Figure \ref{fig:eq_path_Hopf} (a) correspond to two distinct points on this curve). Stable limit cycles are observed numerically in the region bounded above this curve. 

\begin{figure}[t!]
\centering
\subfigure[]{\includegraphics[scale=0.5]{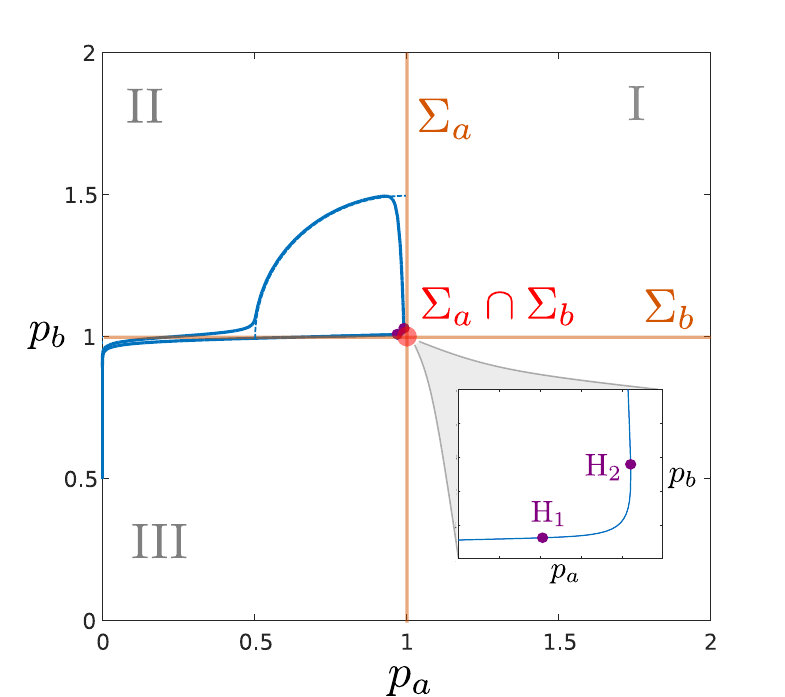} \ \ }
\subfigure[]{\includegraphics[scale=0.64]{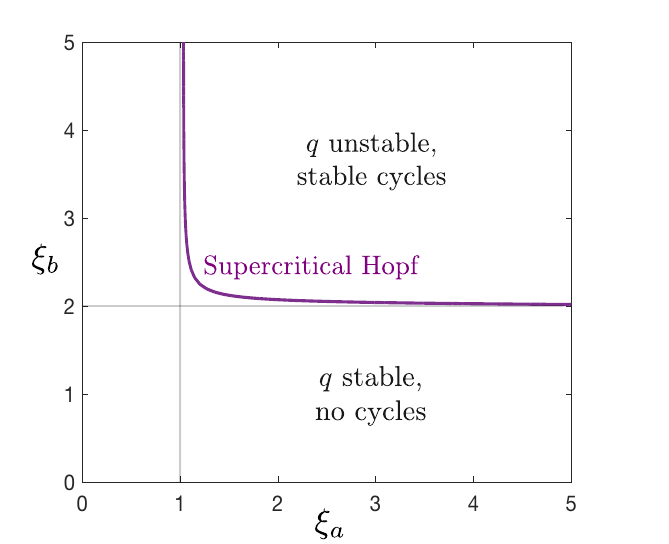}}
\caption{In (a): The diagram obtained by performing the same numerical computations as in Figure \ref{fig:eq_path}, except with $\eps = 5 \times 10^{-3}$ instead of $5 \times 10^{-5}$ (so that $\mu=\frac12$). This time we find two supercritical Hopf bifurcations using MatCont \cite{MATCONT} close to the (projection of the) intersection $\Sigma_a \cap \Sigma_b$; these are labelled H$_1$ and H$_2$ and indicated by purple discs in the zoom/inset. The first Lyapunov coefficients associated with the Hopf bifurcations H$_1$ and H$_2$ were calculated numerically in MatCont to be $l_1 \approx -0.07676$ and $l_1 \approx -0.2384$ respectively, indicating that both bifurcations are supercritical. In (b): The Hopf curve in $(\xi_a, \xi_b)$-space, shown in purple, appears to asymptote along $\xi_a = 1$ and $\xi_b = \gamma = 2$. The Hopf bifurcations H$_1$ and H$_2$ correspond to intersections of this curve with the circle $\mathbb S$ shown in Figure \ref{fig:eq_path} (right).}
\label{fig:eq_path_Hopf}
\end{figure}

The Hopf bifurcations are in agreement with Theorem \ref{thm:thm0}, see also Section \ref{sec:kappa} and Lemma \ref{lem:Hopf} specifically. In fact, with $\gamma=2$ and $\delta=3$ we obtain that 
\begin{align*}
\alpha = \frac{8}{9},
\end{align*}
recall \eqref{eq:alpha}. Therefore with the values in Figure \ref{fig:eq_path}, we have 
\begin{align*}
\mu = 5\times 10^{-3} < \alpha \sigma = \frac{8}{9}\times 10^{-2},
\end{align*}
and the absence of a Hopf bifurcation is therefore consistent with assertion \ref{kappa} of Theorem \ref{thm:thm0}.  In contrast, with the values in Figure \ref{fig:eq_path}, we have
\begin{align*}
\mu = \frac12 > \alpha \sigma = \frac{8}{9}\times 10^{-2},
\end{align*}
and the existence of a Hopf bifurcation is therefore consistent with assertion \ref{kappa} of Theorem \ref{thm:thm0}. \rsp{Recall also Figure \ref{fig:lcnolc}, with (a) and (b) corresponding to the first case (where $\mu$ is too small) with convergence to a steady state and (c) and (d) corresponding to the second case with a visible limit cycle. Notice here that the values of $(\xi_a,\xi_b)=(1.3536,2.3536)$ used in Figure \ref{fig:lcnolc} correspond to the region $2$ of the circle $\mathbb S$ in Figure \ref{fig:eq_path} (b) (it is indicated by a red dot). }
\rsp{More broadly}, the observed existence of stable limit cycles in the region bounded above the curve in Figure \ref{fig:eq_path_Hopf} (b) is consistent with assertion (viii) of Proposition \ref{prop:PWS_dynamics} and Theorem \ref{thm:thm0}. Indeed, by assertion \ref{existence} of Theorem \ref{thm:thm0} we can reduce the system to a $2$-dimensional locally invariant manifold $\mathcal M_{\mu}(\sigma)$ and for the values in the region bounded above the curve in Figure \ref{fig:eq_path_Hopf} (b) the equilibrium is an unstable focus, see Section \ref{sec:kappa}. We can therefore obtain existence of a  limit cycle \SJJ{using the} Poincar\'e-Bendixson \SJJ{theorem}. In fact, assertion (viii) of Proposition \ref{prop:PWS_dynamics} shows that the limit cycles are $o(1)$-close to $(p_a,p_b)=(1,1)$ in the $(p_a,p_b)$-plane with respect to $\mu,\sigma\to 0$. 

\rsp{Finally, in Figure \ref{fig:nonexistence} we show the result of numerical computations of \eqref{eq:main0} for the parameter values in \eqref{eq:para} but with $\sigma=2\times 10^{-3}$ and three different values of $\eps$. In pink: $\eps =2 \times 10^{-2}$, in blue: $\eps=5\times 10^{-3}$ and finally in cyan: $\eps = 8\times 10^{-4}$, corresponding to $\mu=10$, $\mu=2.5$, $\mu=0.4$, respectively. The initial condition is again $(r_a,r_b,p_a,p_b)(0)=(0,0,0.5,0.5)$ for all values of $\eps$. In particular, in Figure \ref{fig:nonexistence} (a) we show a projection onto the $(p_b,r_a)$-plane. This diagram illustrates Assertion \ref{nonexistence} of Theorem \ref{thm:thm0}: the distance from the QSSR approximation is measured by $\mu$. For $\mu=10$ (cyan) there is a large separation from the QSSR approximation. For smaller values of $\mu$ (like $\mu=0.4$ in cyan) the curve follows the QSSR approximation more closely. In Figure \ref{fig:nonexistence}, we show the projection on to the $(p_a,p_b)$-plane. Notice that there is no visible separation between three different curves before the first intersection of $p_b=1$, but beyond this intersection, due to the separation of the QSSR approximation near $p_b=1$ for large values of $\mu$, the trajectories start to deviate, also in the $(p_a,p_b)$-plane. }
\begin{figure}[t!]
\centering
\subfigure[]{\includegraphics[scale=0.5]{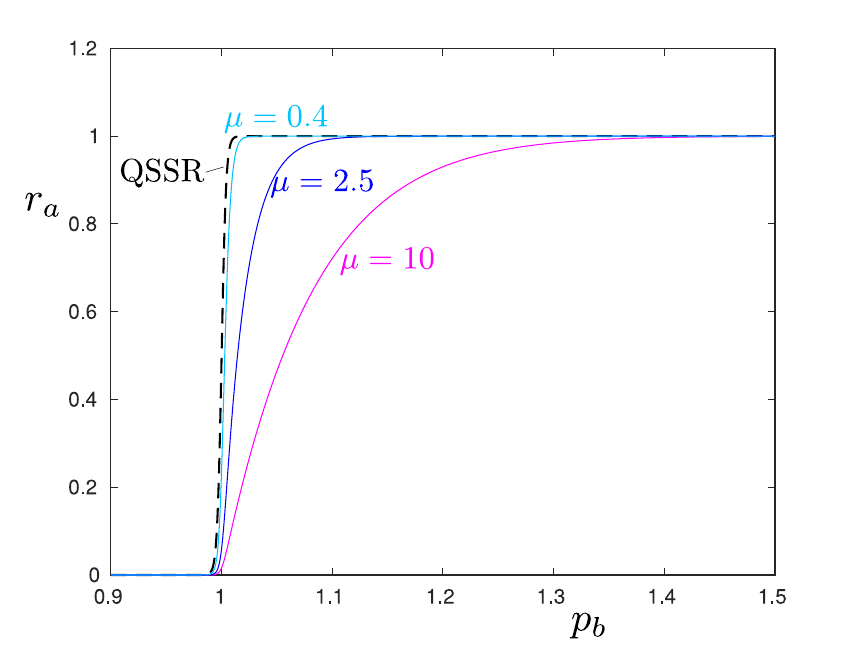}  }
\subfigure[]{\includegraphics[scale=0.5]{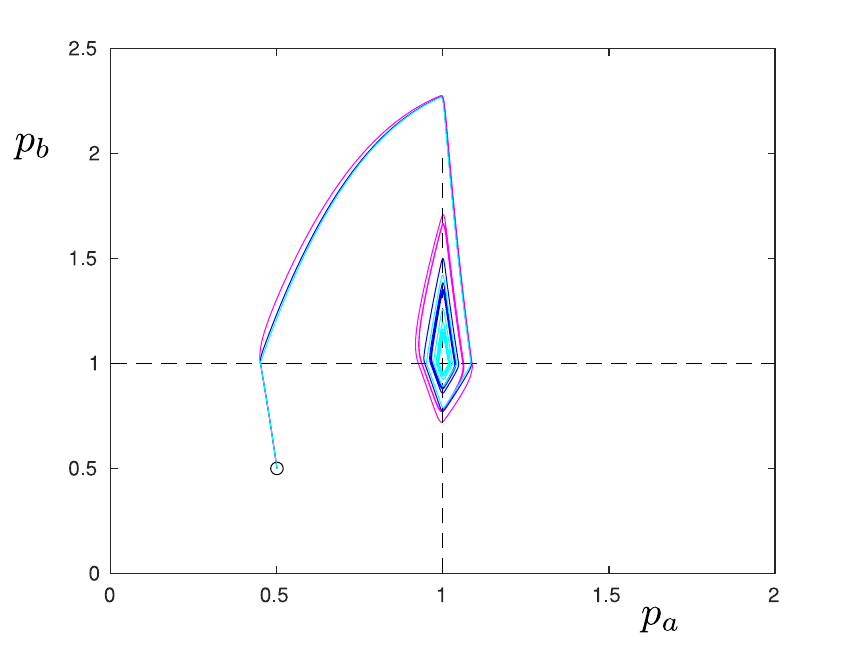} }	
\caption{\rsp{In (a): Projection onto the $(p_b,r_a)$-plane for three different values of $\mu$. The dotted line indicates the QSSR approximation. In (b): Projection onto the $(p_a,p_b)$ for the same values of $\mu$ as in (a).  }}
\label{fig:nonexistence}
\end{figure}

\section{Discussion and outlook}
\label{sec:summary_and_outlook}

The analysis and results in this article show, in the context of the activator-inhibitor model \eqref{eq:main0}, that the asymptotic behaviour of GRN models in general form \eqref{eq:grn_network} as $(\eps, \sigma) \to (0,0)$ may depend in an important way on \textit{how} this limit is taken. Theorem \ref{thm:thm0} assertion \ref{existence} asserts the existence of a `global' $2$-dimensional, attracting invariant manifold $\mathcal M_\mu(\sigma)$, upon which the dynamics are governed by \SJJ{a} reduced protein-only system, 
whenever $0 < \mu = \eps / \sigma \ll 1$, i.e.~when $0 < \eps \ll \sigma \ll 1$ and $0<\sigma<\sigma_0$; assertion \ref{nonexistence} implies that $0<\mu\ll 1$ is also necessary for a global reduction.
%
At the same time, assertion \ref{kappa} of Theorem \ref{thm:thm0} shows that even when  $0 < \mu = \eps / \sigma \ll 1$ holds true, the dynamics on $\mathcal M_\mu(\sigma)$ still depends on the relative size of $\mu$ and $\sigma$. In particular, Hopf bifurcations only occur when $\mu>\alpha \sigma$, see Figure \ref{fig:blow-up0} and Proposition \ref{prop:kappa} for a precise statement; in this parameter region we therefore conclude that the QSSR \SJJ{system} \eqref{eq:M0red}, which \SJJ{does} not support Hopf bifurcations \SJJ{or limit cycles}, recall Lemma \ref{lemma:qssr}, is not a faithful representation.  The basic reason for this `disagreement', is that the higher order contributions in system \eqref{eq:main0} on ${\mathcal M_\mu(\sigma)}$, which are omitted in system \eqref{eq:M0red}, play an important role in determining the qualitative dynamics, see also Remark \ref{rem:Hamiltonian}.

\rsp{It is worthy to review the implications of our findings in the broader context of model reduction for the original high-dimensional and highly nonlinear GRN model defined by system \eqref{eq:grn_network}. We saw that the partial and independent (dimensionality and nonlinearity) reductions represented in Figure \ref{fig:model_reduction} correspond to independent singular limits in system \eqref{eq:grn_network}, namely, the limits $\eps \to 0$ and $\sigma \to 0$ respectively. The reduced model when both $\eps, \sigma \to 0$ is non-unique, because it depends upon the way in which the combined limit is taken. The answer to the question ``which reduced model is the right one?", depends on the decomposition of the $(\eps,\sigma)$-space and, in this case, Theorem \ref{thm:thm0} provides the means for determining the right reduced model because it provides this decomposition (as shown in Figure \ref{fig:par_planes}). In particular, the answer in this case will hinge upon the size of $\mu = \eps / \sigma$. We believe that the problem of non-uniqueness when it comes to the combined reduction shown in Figure \ref{fig:model_reduction} for system \eqref{eq:grn_network} in general, translates into a question about the decomposition of the small parameter space.}

It seems natural to ask about the extent to which the geometric analysis for the `representative', low-dimensional model \eqref{eq:main0} presented herein, can be extended to the study of the $2N$-dimensional GRN models defined by \eqref{eq:grn_network} more generally \rsp{and, in particular, provide some insight into the small parameter space decomposition}. Given that the particular form of the equations did not play an important role in the proof of assertion \ref{existence} \SJJ{in} Theorem \ref{thm:thm0} \SJJ{in particular}, it seems natural to ask whether a counterpart can be proven for \SJJ{the $2N$-dimensional network} \eqref{eq:grn_network}. 
These and other related problems are part of on-going work.

\section*{Acknowledgements}

SJ was partially supported by European Union Marie Sk{\l}odowska-Curie Postdoctoral Fellowship grant 101103827. SJ would also like to thank KUK and the Department of Applied Mathematics and Computer Science for their hospitality during his secondment at the Technical University of Denmark, during which time much of the work presented herein was done.

\section*{Data Availability}

No data was created or analysed in this work.

\bibliographystyle{siam}
\bibliography{references}

\newpage 
\begin{appendix}
\section{Derivation of the model \eqref{eq:main0}}
\label{app:model}

The $4$-dimensional activator-inhibitor model considered in \cite{Polynikis2009,Widder2007} can be written as
\begin{equation}
\label{eq:main_original}
\begin{split}
r_a' &= m_a h^+(p_b; \theta_b, n_b) - \gamma_a r_a , \\
r_b' &= m_b h^-(p_a; \theta_a, n_a) - \gamma_b r_b , \\
p_a' &= \eps \left( k_a r_a - \delta_a p_a \right) , \\
p_b' &= \eps \left( k_b r_b - \delta_b p_b \right) ,
\end{split}
\end{equation}
where, as in the general network \eqref{eq:grn_network}, $r_i \geq 0$ represents the concentration of transcribed mRNA associated with a particular gene $i$, $p_i \geq 0$ represents the concentration of the corresponding translated protein, $\eps, n_i, m_i, \theta_i, \gamma_i, k_i, \delta_i > 0$ are parameters ($i \in \{a,b\}$) and the dash denotes differentiation with respect to time $t$. 

The notation and formulation in system \eqref{eq:main_original} agrees with \cite[eqn.~(41)]{Polynikis2009}, except for the following two differences:
\begin{enumerate}
\item[(i)] The parameter $\eps$ in system \eqref{eq:main_original} is actually $\eps^{-1}$ in \cite{Polynikis2009}. We prefer to work with a small (as opposed to large) parameter $\eps$;
\item[(ii)] System \eqref{eq:main_original} is written on the fast time-scale $t$, whereas \cite[eqn.~(41)]{Polynikis2009} is written on the slow time-scale $\tau = \eps t$.
\end{enumerate}
In this paper, we assume that $n_a = n_b$ and define
\begin{align}\label{ass:switching}
0 < \sigma := n_a^{-1} = n_b^{-1}.
\end{align}

\begin{lemma}
\label{lem:normal_form}
Given assumption \eqref{ass:switching}, system \eqref{eq:main_original} is equivalent (by rescalings only) to the system
\begin{equation}
\label{eq:main}
\begin{split}
\tilde r_a' &= h^+\left( \tilde p_b; 1, \sigma^{-1} \right) - \tilde r_a , \\
\tilde r_b' &= h^-\left( \tilde p_a; 1, \sigma^{-1} \right) - \gamma \tilde r_b , \\
\tilde p_a' &= \tilde \eps \left( \xi_a \tilde r_a - \tilde p_a \right) , \\
\tilde p_b' &= \tilde \eps \delta \left( \xi_b \tilde r_b - \tilde p_b \right) ,
\end{split}
\end{equation}
where the dash now denotes differentiation with respect to $\tilde t = \gamma_a t$, and
\begin{equation}
\label{eq:parameters_rescaled}
\gamma = \frac{\gamma_b}{\gamma_a} , \qquad 
\xi_a = \frac{k_a m_a}{\gamma_a \theta_a} , \qquad 
\xi_b = \frac{k_b m_b}{\gamma_a \theta_b} , \qquad 
\delta = \frac{\delta_b}{\delta_a} , \qquad
\tilde \eps = \frac{\delta_a}{\gamma_a} \eps .
\end{equation}
\end{lemma}

\begin{proof}
Applying the time rescaling $t = \gamma_a^{-1} \tilde t$ and the coordinate rescaling
\[
r_a = \frac{m_a}{\gamma_a} \tilde r_a, \qquad 
r_b = \frac{m_b}{\gamma_a} \tilde r_b, \qquad 
p_a = \theta_a \tilde p_a , \qquad
p_b = \theta_b \tilde p_b ,
\]
yields the desired result.
\end{proof}
Upon dropping the tildes in \eqref{eq:main} we obtain \SJJ{system} \eqref{eq:main0}.

\section{\rsp{Proof of Assertion \ref{existence} of Theorem~\ref{thm:thm0}} }
\label{app:blow-up}

In order to prove Assertion \ref{existence} of Theorem~\ref{thm:thm0}, we consider the extended system
\begin{equation}
\label{eq:main_extended}
\begin{split}
r_a' &= {\phi}( \sigma^{-1}\log p_b) - r_a , \\
r_b' &= 1-{\phi}( \sigma^{-1}\log p_a) - \gamma r_b , \\
p_a' &= \mu \sigma p_a G_a(r_a,\log p_a ), \\
p_b' &= \mu \sigma  p_b G_b(r_b,\log p_b) , \\
\sigma' &= 0 ,
\end{split}
\end{equation}
which is obtained from system \eqref{eq:main0} with $\eps=\mu\sigma$ after appending the trivial equation $\sigma' = 0$. Here we have also simplified the expressions for the Hill functions by defining
\begin{align}\label{eq:Hpm}
{\phi}(x): = \frac{\me^{x}}{1 + \me^{x}},
\end{align}
and \SJ{introducing}
\begin{equation}\label{eq:fafb}
\begin{aligned}
G_a(x,y) := \xi_a x \me^{-y} - 1 , \qquad 
G_b(x,y) := \delta( \xi_b x \me^{-y} - 1 ).
\end{aligned}
\end{equation}
This will help to simplify the notation in charts below. In fact, assertion \ref{existence} of Theorem~\ref{thm:thm0} holds true for any smooth functions $G_a$ and $G_b$. Only assertions \ref{kappa}--\ref{nonexistence} use \eqref{eq:fafb}. 
%
The system \eqref{eq:main_extended} loses smoothness along $\Sigma\times \{0\}$, recall \eqref{eq:Sigma} and \eqref{eq:Sigmaab}. In the following, we will resolve this lack of smoothness through blow-up and show that the QSSR \SJJ{constraints}
\begin{equation}\label{eq:M00}
\begin{aligned}
r_a &={\phi}(\sigma^{-1} \log p_b),\qquad
r_b=\gamma^{-1}(1-{\phi}(\sigma^{-1} \log p_a)),
\end{aligned}
\end{equation}
when written in the \SJJJ{blown-up space}, define an attracting and normally hyperbolic critical manifold for $\mu=0$. 
\subsection{Spherical blow-up along the intersection $\Sigma_a \cap \Sigma_b$}

First, we resolve the most degenerate set $(\Sigma_a \cap \Sigma_b) \times \{0\}\subset \Sigma \times \{0\}$, corresponding to $(p_a,p_b,\sigma)=(1,1,0)$, by blowing it up to a `cylinder of spheres'. 
In particular, we define 
\begin{align*}
\mathbb P &:= \left\{(r_a,r_b,p_a,p_b,\sigma)\,:\,r_a,r_b,p_a,p_b,\sigma \ge 0\right\},
\end{align*}
so that $\mathbb P\setminus \Sigma$ is the phase-space of \SJJ{system} \eqref{eq:main_extended}, and \SJ{write}
\begin{align*}	 
\overline{\mathbb P} &:=\left\{(r_a,r_b,\eta,(\bar u,\bar v,\bar \sigma))\,: r_a,r_b,\eta\ge 0,\,(\bar u,\bar v,\bar \sigma)\in \mathbb H^2\right\} ,
\end{align*}	
\SJJ{where}
\[
\mathbb H^2 :=\{(\bar u,\bar v,\bar \sigma)\in \mathbb R^3\,:\,\bar u^2+\bar v^2+\bar \sigma^2=1,\,\bar \sigma\ge 0\}
\]
is the $\bar \sigma\ge 0$-hemisphere of \SJJ{the sphere} $\mathbb S^2$.
Then, following \cite{Frieder2018}, the relevant blow-up transformation $\overline \Phi:\overline{\mathbb P}\rightarrow \mathbb P$ fixes $r_a$ and $r_b$ and takes
\begin{align}\label{eq:bu0}
(\eta,(\bar u, \bar v, \bar \sigma))\mapsto 
\begin{cases}
p_a = \me^{\eta \bar u} , \\
p_b = \me^{\eta \bar v} , \\
\ \sigma = \eta \bar \sigma,
\end{cases}\quad \eta\ge 0,(\bar u, \bar v, \bar \sigma)\in \mathbb H^2.
\end{align}
Therefore the pre-image of $(\Sigma_a\cap \Sigma_b)\times \{0\}$ under $\overline{\Phi}$ is $r_a,r_b\ge 0,\,\eta=0,\,(\bar u, \bar v, \bar \sigma) \in \mathbb H^2$. 
Notice also that under the substitution of \eqref{eq:bu0}, \SJ{the defining equations for the QSSR manifold in \eqref{eq:M0} become} 
\begin{equation}\label{eq:M00bu}
r_a ={\phi}(\bar v \bar \sigma^{-1}),\qquad
r_b=\gamma^{-1} (1-{\phi}(\bar u \bar \sigma^{-1})),
\end{equation}
which \SJ{are} well-defined for $(\bar u, \bar v, \bar \sigma) \in \mathbb H^2\cap \{\bar \sigma>0\}$. In fact, \eqref{eq:M00bu} has a well-defined smooth extension to the equator circle $\mathbb H^2 \cap \{\bar \sigma=0\}$ away from the degenerate points $(\bar u,\bar v,\bar \sigma)=(\pm 1,0,0),(0,\pm 1,0)$ defined by
\begin{equation}\nonumber
\begin{aligned}
r_a &=\begin{cases}
1 & \overline v>0,\,\bar \sigma=0,\\
0 & \overline v<0,\,\bar \sigma=0,
\end{cases}\quad
r_b=\begin{cases}
0 & \overline u>0,\,\bar \sigma=0,\\
\gamma^{-1} & \overline u<0,\,\bar \sigma=0.
\end{cases}
\end{aligned}
\end{equation}
Here we have used the definition \eqref{eq:Hpm}. We therefore define the corresponding degeneracies as subsets of $\overline{\mathbb P}$:
\begin{align*}
\overline \Sigma = \bigcup_{\substack{i\in \{a,b\}\\
j\in \{\pm\}}} \overline \Sigma_i^j,
\end{align*}
where	
\begin{align*}
\overline \Sigma_{a}^\pm &= \{(r_a,r_b,\eta,(\bar u,\bar v,\bar \sigma))\in \overline{\mathbb P}\,:\,(\bar u,\bar v,\bar \sigma)=(\pm 1,0,0)\},\\
\overline \Sigma_{b}^\pm &= \{(r_a,r_b,\eta,(\bar u,\bar v,\bar \sigma))\in \overline{\mathbb P}\,:\,(\bar u,\bar v,\bar \sigma)=(0,\pm 1,0)\},
\end{align*}
respectively. 
Now, let $X$ denote the vector-field defined by \SJJ{system} \eqref{eq:main_extended} on \SJJ{$\mathbb P \setminus \Sigma$}. 
In Section \ref{sec:Ki}, we will show that the pull-back vector-field $$\overline X:=\Phi^*(X),$$ is well-defined on 
\begin{align}\label{eq:Pmins}
\overline{\mathbb P}\setminus \overline \Sigma,
\end{align}
and that \eqref{eq:M00bu} defines an attracting and normally hyperbolic  critical submanifold $\overline{\mathcal M}$ of \KUK{\eqref{eq:Pmins}} for $\mu=0$.
For this purpose, we will work with an atlas consisting of five different coordinate charts, defined projectively via
\[
K_1 : \bar u = - 1, \quad 
K_2 : \bar \sigma = 1, \quad 
K_3 : \bar u = 1, \quad 
K_4 : \bar v = - 1, \quad
K_5 : \bar v = 1 .
\]
We introduce local coordinates in each chart as follows:
\begin{equation}\label{eq:K2}
\begin{aligned}
K_1 :& \ (p_a, p_b, \sigma) = (\me^{-\eta_1}, \me^{\eta_1 v_1}, \eta_1 \sigma_1),\\
K_2 :& \ (p_a, p_b, \sigma) = (\me^{\eta_2 u_2}, \me^{\eta_2 v_2}, \eta_2), \\
K_3 :& \ (p_a, p_b, \sigma) = (\me^{\eta_3}, \me^{\eta_3 v_3}, \eta_3 \sigma_3),\\
K_4 :& \ (p_a, p_b, \sigma) = (\me^{\eta_4 u_4}, \me^{- \eta_4}, \eta_4 \sigma_4),\\
K_5 :& \ (p_a, p_b, \sigma) = (\me^{\eta_5 u_5}, \me^{\eta_5}, \eta_5 \sigma_5),
\end{aligned}
\end{equation}
where all $\eta_i\ge 0$, $\sigma_i\ge 0$, $u_i\in \mathbb R$, $v_i\in \mathbb R$.
Diffeomorphic change of coordinate maps between overlapping coordinate charts can be derived using these expressions. Given that there is a large number of them, we do not state them explicitly here.

\subsubsection{Analysis in charts $K_i$}\label{sec:Ki}

The equations in the scaling chart $K_2$ are given by
\begin{equation}
\label{eq:K2_system}
\begin{split}
r_a' &= {\phi}(v_2) - r_a , \\
r_b' &= 1-{\phi}(u_2) - \gamma r_b , \\
u_2' &= \mu G_a(r_a, \eta_2 u_2) , \\
v_2' &= \mu G_b(r_b, \eta_2 v_2) ,\\
\eta_2'&=0.
\end{split}
\end{equation}
This is the local form $X_2$ of $\overline X$ in chart $K_2$.  
\SJJ{System \eqref{eq:K2_system} is} smooth and slow-fast in standard form with respect to $\mu \to 0$, depending regularly upon $\eta_2=\sigma\ge 0$. In particular, the local form $\mathcal M_{2}$ of $\overline{\mathcal M}$, defined by
\begin{align}
r_a &={\phi}(v_2),\qquad
r_b=\gamma^{-1} (1-{\phi}(u_2)),\label{eq:M20}
\end{align}
is an attracting and normally hyperbolic critical manifold of \SJJ{\eqref{eq:K2_system}$|_{\mu = 0}$} with nontrivial eigenvalues $-1$ and $-\gamma<0$.
We will use the following result later on.
\begin{lemma}\label{lem:slow_manifolds_2}
Fix any $k\in \mathbb N$, $\eta_{20}>0$, and any compact \SJJ{connected} domain $K_2\subset \mathbb R^2$. \SJJ{There exists a $\mu_0 > 0$ such that} the compact submanifold $\mathcal M_{0,2}\subset \mathcal M_{2}$, defined by \eqref{eq:M20} and $(u_2,v_2)\in K_2$, persists as a locally invariant attracting $C^k$-smooth slow manifold $\mathcal M_{\mu,2}$  of the graph form
\begin{equation}\label{eq:Mmu2}
\begin{aligned}
r_a &= {\phi}(v_2) + \mu \Omega_a(u_2, v_2, \eta_2) + \mathcal O(\mu^2) ,\\ 
r_b &= 1-{\phi}(u_2) + \mu \Omega_b(u_2, v_2, \eta_2) + \mathcal O(\mu^2) ,
\end{aligned}
\end{equation}
for $(u_2,v_2)\in K_2$, $\eta_2\in [0,\eta_{20}]$ and all \SJJ{$\mu \in [0, \mu_0)$}. 
In particular,
\begin{equation}\label{eq:omegaab}\begin{aligned}
\Omega_a(u_2, v_2, \eta_2) &= - \delta {\phi}(v_2) (1-{\phi}(v_2)) G_b \left( \gamma^{-1} (1-{\phi}(u_2)) , \eta_2 v_2 \right) , \\ 
\Omega_b(u_2, v_2, \eta_2) &= \gamma^{-2} {\phi}(u_2) (1-{\phi}(u_2)) G_a \left( {\phi}(v_2), \eta_2 u_2 \right), 
\end{aligned}
\end{equation} and the $\mathcal O(\mu^2)$-terms are all $C^k$-smooth with respect to $(u_2,v_2)\in K_2,\,\eta_2\in [0,\eta_{20}]$, and \SJJ{$\mu \in [0, \mu_0)$}. 
\end{lemma}
\begin{proof}
The existence of $\mathcal M_{\mu,2}$ follows directly from Fenichel theory, see e.g. \cite[Thm.~2]{Jones1995}. The expansions in \eqref{eq:Mmu2} and \eqref{eq:omegaab} are  consequences of simple calculations, the details of which we leave out for simplicity.
\end{proof}

\medskip

Since the equations in the phase-directional charts $K_1$, $K_3$, $K_4$ and $K_5$ have a similar algebraic structure, we shall consider them simultaneously. In order to do so, we introduce the following additional notation:
\[
s_i = 
\begin{cases}
+1 & \text{if} \ i \in \{3,5\} , \\
-1 & \text{if} \ i \in \{1,4\} .
\end{cases}
\]
This allows us to write the equations in $K_1$ and $K_3$ collectively as
\begin{equation}
\label{eq:bar_u_eqns_sphere}
\begin{split}
r_a' &= {\phi}(v_i \sigma_i^{-1}) - r_a , \\
r_b' &= 1-{\phi}(s_i \sigma_i^{-1})) - \gamma r_b , \\
\eta_i' &= \mu s_i \eta_i \sigma_i G_a(r_a, s_i \eta_i) , \\
v_i' &= \mu \sigma_i \left[ G_b(r_b, \eta_i v_i) - s_i v_i G_a(r_a, s_i \eta_i) \right] , \\
\sigma_i' &= - \mu s_i \sigma_i^2 G_a(r_a, s_i \eta_i) ,
\end{split}
\end{equation}
where $i \in \{1,3\}$ is the index associated with the coordinate chart $K_i$. Similarly, we can write the equations in $K_4$ and $K_5$ collectively as
\begin{equation}
\label{eq:bar_v_eqns_sphere}
\begin{split}
r_a' &= {\phi}(s_i \sigma_i^{-1}) - r_a , \\
r_b' &= 1-{\phi}(u_i \sigma_i^{-1}) - \gamma r_b , \\
u_i' &= \mu \sigma_i \left[ G_a(r_a, \eta_i u_i) - s_i u_i G_b(r_b, s_i \eta_i ) \right] , \\
\eta_i' &= \mu s_i \eta_i \sigma_i G_b(r_b, s_i \eta_i) , \\
\sigma_i' &= - \mu s_i \sigma_i^2 G_b(r_b, s_i \eta_i) ,
\end{split}
\end{equation}
where $i \in \{4,5\}$. 
These are the local forms $X_i$ of $\overline X$ in the charts $K_i$, $i\in \{1,3,4,5\}$, respectively. 
$X_1$ and $X_3$ \SJJ{(}given by \eqref{eq:bar_u_eqns_sphere}\SJJ{)} are smooth everywhere \SJ{except} along the $3$-dimensional hyperplanes
\[
\Sigma_{b,i} = \left\{ (r_a, r_b, \eta_i, 0, 0) : r_a, r_b, \eta_i \geq 0 \right\} ,
\]
defined by $v_i=\sigma_i =0$ for $i\in \{1,3\}$, respectively; these sets are the local versions of $\overline \Sigma_b^{\pm}$ and correspond to the left and right branches of $\Sigma_b \times \{0\}
$ after blow-down (i.e. upon application of the map $\overline \Phi$) respectively, see Figure \ref{fig:blow-up0} (a) and (b). Moreover, away from $\Sigma_{b,i}$ the local versions $\mathcal M_{i}$, $i\in \{1,3\}$, of $\mathcal M$, defined by 
\begin{align*}
r_a = {\phi}(v_i \sigma_i^{-1}),\qquad r_b =\gamma^{-1} (1-{\phi}(s_i \sigma_i^{-1})),
\end{align*}
are attracting and normally hyperbolic critical manifolds of \eqref{eq:bar_u_eqns_sphere} with nontrivial eigenvalues $-1$ and $-\gamma<0$.

Similarly, $X_4$ and $X_5$ \SJJ{(}given by \eqref{eq:bar_v_eqns_sphere}\SJJ{)} are smooth everywhere \SJ{except} along the $3$-dimensional hyperplanes
\[
\Sigma_{a,i} = \left\{ (r_a, r_b, 0, \eta_i, 0) : r_a, r_b, \eta_i \geq 0 \right\} ,
\]
defined by $u_i=\sigma_i=0$ for $i\in\{4,5\}$, respectively; these sets are the local versions of $\overline \Sigma_a^{\pm}$ and correspond to the lower and upper branches of $\Sigma_a \times \{0\}$ after blow-down respectively. We refer again to Figure \ref{fig:blow-up0} (a) and (b). Finally, away from $\Sigma_{a,i}$ the local versions $\mathcal M_{i}$, $i\in \{4,5\}$, of $\mathcal M$, defined by 
\begin{align*}
r_a = {\phi}(s_i \sigma_i^{-1}),\qquad r_b =\gamma^{-1} (1-{\phi}(u_i \sigma_i^{-1})),
\end{align*}
are attracting and normally hyperbolic critical manifolds of \eqref{eq:bar_v_eqns_sphere} with nontrivial eigenvalues $-1$ and $-\gamma<0$. 


\subsection{Cylindrical blow-ups along $\overline \Sigma$}

We now turn to the remaining degeneracies along $\overline \Sigma$. Consider first the subsets $\SJJ{\overline \Sigma_b^\pm} \subset \overline \Sigma$ and define the following sets
\begin{align*}
{\overline{\mathbb P}}_b^\pm &:=\overline{\mathbb P}\cap \{\bar u\gtrless 0\},
\end{align*}
respectively, and 
\begin{align*}
\overline{\overline{\mathbb Q}}& := \left\{(r_a,r_b,\eta,\rho,(\bar{\bar u},\bar{\bar \sigma}))\,:\,r_a,r_b,\eta,\rho\ge 0,\,(\bar{\bar v},\bar{\bar \sigma}) \in \mathbb H^1\right\},
\end{align*}
where
\[
\mathbb H^1:=\{(\bar{\bar v},\bar{\bar \sigma})\in \mathbb R^2\,:\,\bar{\bar v}^2+\bar{\bar \sigma}^2=1,\,\bar{\bar \sigma}\ge 0\}
\]
\SJJ{is the $\bar{\bar \sigma} \geq 0$ half of the unit circle $\mathbb S^1$. Now we} apply the blow-up transformations $\overline{\overline \Phi}_b^\pm:\overline{\overline{\mathbb Q}} \rightarrow {\overline{\mathbb P}}_b^\pm$, which fix $r_a,r_b, \eta$ and take
\begin{align} \label{eq:barbarPhiapm}
(\rho,(\bar{\bar u},\bar{\bar \sigma}))\mapsto \begin{cases}
\bar v \bar u^{-1}=\pm \rho \bar{\bar v},\\
\bar \sigma \bar u^{-1}=\pm \rho \bar{\bar \sigma},
\end{cases} \quad \rho\ge 0,\,(\bar{\bar v},\bar{\bar \sigma})\in \mathbb H^1,
\end{align}
respectively.
The preimages of $\overline \Sigma_b^\pm$ under $\overline{\overline \Phi}_b^\pm$ are the cylinders defined by $\overline{\overline{\mathbb Q}} \cap \{\rho=0\}$. 
Notice also that under the substitutions of \eqref{eq:barbarPhiapm}, \SJJ{the QSSR constraints in} \eqref{eq:M00bu} become
\begin{align}\label{eq:M00bua}
r_a ={\phi}(\bar{\bar v}\bar{\bar \sigma}^{-1}),\qquad
r_b=\gamma^{-1} {(1-{\phi}( \rho^{-1}\bar{\bar \sigma}^{-1} ))},
\end{align}
respectively.
There are two cases, corresponding to $\pm$, and they are both well-defined on $\overline{\overline{\mathbb Q}}$. Indeed, in polar coordinates $(\bar{\bar v},\bar{\bar \sigma})=(\cos {\nu},\sin{\nu})$, ${\nu}\in [0,\pi]$, we have
\begin{align*}
{\phi}(\bar{\bar v}\bar{\bar \sigma}^{-1}) = \begin{cases}
1 & \text{if}\,{\nu}=0\\
\frac{e^{\cot {\nu}}}{1+e^{\cot {\nu}}} &\text{if}\,{\nu} \in (0,\pi)\\
0 & \text{if}\,{\nu} = \pi
\end{cases},\qquad {\phi}(-\bar{\bar v}\bar{\bar \sigma}^{-1})= \begin{cases}
0 & \text{if}\,{\nu}=0\\
\frac{e^{-\cot {\nu}}}{1+e^{-\cot {\nu}}} &\text{if}\,{\nu} \in (0,\pi)\\
1 & \text{if}\,{\nu} = \pi
\end{cases},
\end{align*}
respectively. \SJJ{These} expressions are $C^\infty$-flat at the end points ${\nu}=0$ and ${\nu}=\pi$.

The degeneracies along $\overline \Sigma_a^\pm$ are handled in the same way. First, we define 
\begin{align*}
{\overline{\mathbb P}}_a^\pm &:= {\overline{\mathbb P}}\cap \{\bar v\gtrless 0\},
\end{align*}
and apply the blow-up transformations $\overline{\overline \Phi}_a^\pm:\overline{\overline{\mathbb Q}}\rightarrow {\overline{\mathbb P}}_a^\pm$ which fix $r_a,r_b, \eta$ and take
\begin{align} \label{eq:barbarPhibpm}
(\rho,(\bar{\bar u},\bar{\bar \sigma}))\mapsto \begin{cases}
\bar u \bar v^{-1}=\pm \rho \bar{\bar u},\\
\bar \sigma \bar v^{-1}=\pm \rho \bar{\bar \sigma},
\end{cases}\quad \rho\ge 0,\,(\bar{\bar u},\bar{\bar \sigma})\in \mathbb H^1,
\end{align}
respectively.
Under the substitutions of \eqref{eq:barbarPhibpm}, \SJJ{the QSSR constraints in} \eqref{eq:M00bu} become
\begin{align}\label{eq:M00bub}
r_a = {\phi}(\rho^{-1}\bar{\bar \sigma}^{-1}  ),\qquad
r_b=\gamma^{-1}(1-{\phi}(\bar{\bar u} \bar{\bar \sigma}^{-1})),
\end{align}
respectively.
There are again two cases, corresponding to $\pm$, and they are both well-defined on $\overline{\overline{\mathbb Q}}$, including on the cylinder $\overline{\overline{\mathbb Q}} \cap \{\rho=0\}$ which corresponds to the blow-up of $\overline \Sigma_a^\pm$ by the inverse process of $\overline{\overline \Phi}_a^\pm$.

In Sections \ref{sec:K13j} and \ref{sec:K45j}, we will show that the pull-back vector-fields
\begin{align*}
\overline{\overline X}_b^\pm = (\overline{\overline \Phi}^\pm_b)^*(\overline X\vert_{\overline{\mathbb P}_b^\pm}),\quad \overline{\overline X}_a^\pm = (\overline{\overline \Phi}^\pm_a)^*(\overline X\vert_{\overline{\mathbb P}_a^\pm}),
\end{align*}
have well-defined smooth extensions to $\overline{\overline{\mathbb Q}} \cap \{\rho=0\}$. Moreover, we will show that \eqref{eq:M00bua} and \eqref{eq:M00bub} define attracting 
and normally hyperbolic critical manifolds $\overline{\overline{\mathcal M}}_{b}^\pm$ 
and $\overline{\overline{\mathcal M}}_{b}^\pm$ of $\overline{\overline X}_b^\pm$ and $\overline{\overline X}_{a}^\pm$, respectively. 
For this purpose, we define atlases of local charts. 
Notice first that in the charts $K_i$, $i\in \{1,3\}$, which cover ${\overline{\mathbb P}}_b^\pm$, the transformations \eqref{eq:barbarPhiapm} take the local forms 
\begin{align*}
(\rho,(\SJJ{\bar{\bar v}},\bar{\bar \sigma}))\mapsto \begin{cases}
v_i = \rho \bar{\bar v},\\
\sigma_i = \rho \bar{\bar \sigma},
\end{cases}\quad i\in \{1,3\},
\end{align*}
respectively. 
We therefore introduce $6$ local coordinate charts which cover the three sides $\bar{\bar \sigma}>0$, $\bar{\bar v}\lessgtr 0$  of each branch of the left (when $i = 1$) and right (when $i = 3$) blow-up cylinders in Figure \ref{fig:blow-up0} (c); see also Figure \ref{fig:charts}. We define them via
\[
\quad K_{i2} : \KUK{\bar{\bar \sigma}}_i = 1, \quad  K_{i4} : \KUK{\bar{\bar v}}_i = -1, \quad 
K_{i5} : \KUK{\bar{\bar v}}_i = 1 ,
\]
respectively, and write the local coordinates as
\begin{equation}\label{eq:K13j}
\begin{aligned}
K_{i2} : (v_i, \sigma_i) &= ( \rho_{i2} v_{i2}, \rho_{i2}) , \\
K_{i4} : (v_i, \sigma_i) &= (- \rho_{i4}, \rho_{i4} \sigma_{i4}) , \\
K_{i5} : (v_i, \sigma_i) &= (\rho_{i3}, \rho_{i5} \sigma_{i5}) ,
\end{aligned}
\end{equation}
where all $\rho_{ij}\ge 0$, $\sigma_{ij}\ge 0$, $v_{i2}\in \mathbb R$. By composing \eqref{eq:K13j} with \eqref{eq:K2}, we can also write the local coordinates in terms of the original variables $(p_a,p_b,\sigma)$ as follows:
\begin{equation}\label{eq:K13j_2}
\begin{aligned}
K_{i2} : (p_a,p_b,\sigma) &= (\me^{s_i \eta_i},\me^{ \eta_i \rho_{i2} v_{i2}},\eta_i \rho_{i2}), \\
K_{i4} : (p_a,p_b,\sigma) &= (\me^{s_i \eta_i},\me^{- \eta_i \rho_{i4}} ,\eta_i \rho_{i4}\sigma_{i4}), \\
K_{i5} : (p_a,p_b,\sigma) &= (\me^{s_i \eta_i},\me^{\eta_i \rho_{i5}} , \eta_i \rho_{i5}\sigma_{i5}).
\end{aligned}
\end{equation}

We proceed analogously for ${\overline{\mathbb P}}_a^\pm$. First, we notice that in the charts $K_i$, $i\in \{4,5\}$, which cover ${\overline{\mathbb P}}_a^\pm$, the transformations \eqref{eq:barbarPhibpm} take the local forms 
\begin{align*}
(\rho,(\bar{\bar u},\bar{\bar \sigma}))\mapsto \begin{cases}
u_i = \rho \bar{\bar u},\\
\sigma_i = \rho \bar{\bar \sigma},
\end{cases}\quad i\in \{4,5\},
\end{align*}
respectively. We again work in $6$ local coordinate charts, which cover the three sides $\bar{\bar \sigma}>0$, $\bar{\bar u}\lessgtr 0$ of each branch of the bottom (when $i = 4$) and top (when $i = 5$) blow-up cylinders emanating from the blow-up sphere in Figure \ref{fig:blow-up0} (c); see also Figure \ref{fig:charts}. We define them via
\[
K_{i1} : \KUK{\bar{\bar u}}_i = -1, \qquad
K_{i2} : \KUK{\bar{\bar \sigma}}_i = 1, \qquad 
K_{i3} : \KUK{\bar{\bar u}}_i = 1 ,
\]
respectively, and write the local coordinates as
\begin{equation}\label{eq:K24j}
\begin{aligned}
K_{i1} : (u_i, \sigma_i) &= (- \rho_{i1}, \rho_{i1} \sigma_{i1}) , \\
K_{i2} : (u_i, \sigma_i) &= (\rho_{i2} u_{i2}, \rho_{i2}) , \\
K_{i3} : (u_i, \sigma_i) &= (\rho_{i3}, \rho_{i3} \sigma_{i3}) ,
\end{aligned}
\end{equation}
where all $\rho_{ij}\ge 0$, $\sigma_{ij}\ge 0$, $u_{i2}\in \mathbb R$. By composing \eqref{eq:K24j} with \eqref{eq:K2}, we can also write the local coordinates in terms of the original variables $(p_a,p_b,\sigma)$ as follows:
\begin{equation}\label{eq:K24j_2}
\begin{aligned}
K_{i1} : (p_a,p_b,\sigma) &= (\me^{- \eta_i \rho_{i1}}, \me^{s_i \eta_i}, \eta_i \rho_{i1} \sigma_{i1}) , \\
K_{i2} : (p_a,p_b,\sigma) &=(\me^{ \eta_i \rho_{i2}u_{i2}}, \me^{s_i \eta_i}, \eta_i\rho_{i2}),\\
K_{i3} : (p_a,p_b,\sigma) &=(\me^{ \eta_i \rho_{i3}}, \me^{s_i \eta_i}, \eta_i \rho_{i3}\sigma_{i3}).
\end{aligned}
\end{equation}


\begin{remark}
\label{rem:notation}
The subscript notation here and throughout this section has been chosen so that the subscripts/indices $1$ and $3$ are associated with `$\bar u = \pm 1$', whereas the subscripts/indices `$4$' and `$5$' are associated with `$\bar v = \pm 1$'. For clarity, this is sketched in Figure \ref{fig:charts}.
\end{remark}

\begin{figure}[t!]
\centering
\includegraphics[scale=0.4]{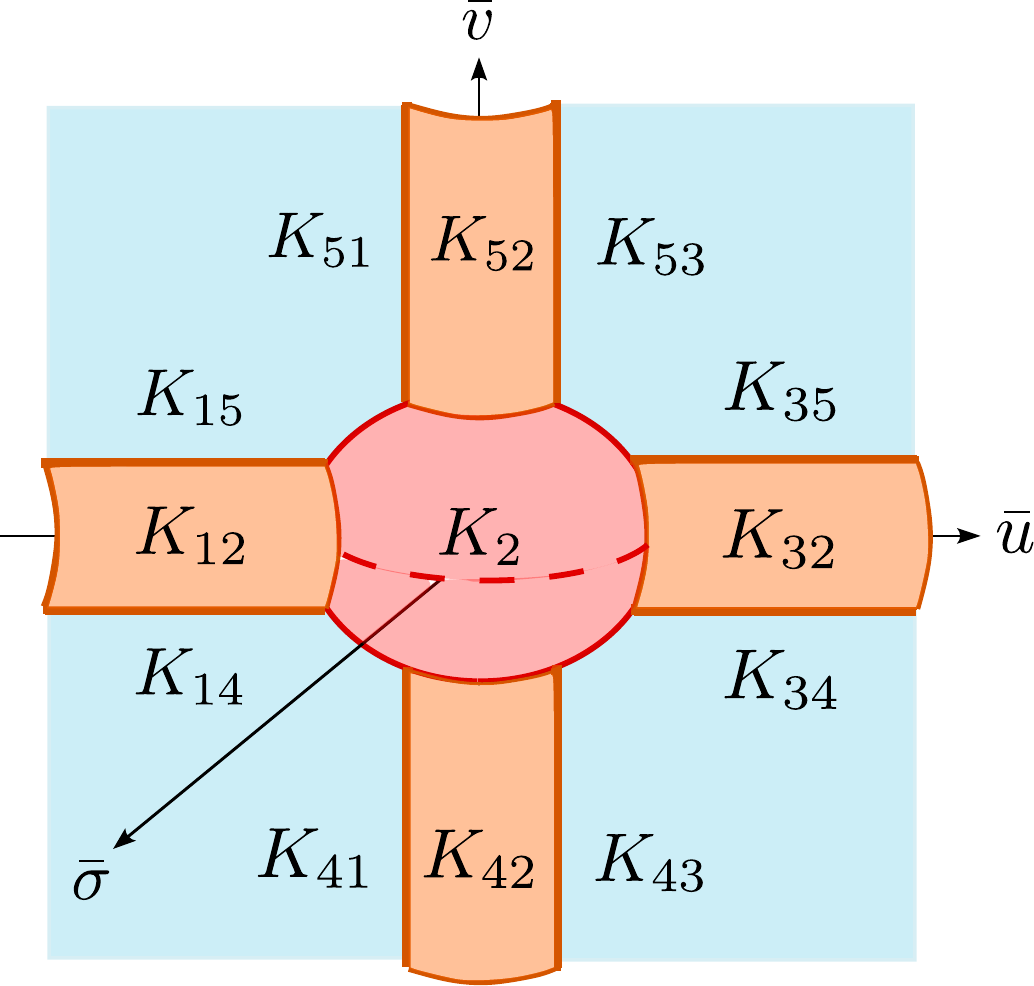}
\caption{Schematic representation of the regions in $(\bar u, \bar v, \bar \sigma)$-space that are covered by the rescaling chart $K_2$ and the $12$ coordinate charts $K_{ij}$ defined in the text; see also Remark \ref{rem:notation}. These $13$ coordinate charts are used in order to cover the (complete) blown-up space.}
\label{fig:charts}
\end{figure}

In combination, $\overline{\overline \Phi}_b^\pm$ and $\overline{\overline \Phi}_a^\pm$ blow up $\overline{\mathbb P}$ to our final blown up space $\overline{\overline{\mathbb P}}$, which has been illustrated in \SJJ{Figures \ref{fig:blow-up0} (c) and \ref{fig:charts}}. The $13$ charts given by\SJ{
\begin{align}\label{eq:allKs}
K_2, \quad  
\left\{ K_{ij} : i \in \{ 1, 3 \} , j \in \{2, 4, 5\} \right\}  \quad \mbox{and} \quad 
\left\{ K_{ij} : i \in \{ 4, 5 \} , j \in \{1, 2, 3\} \right\} ,
\end{align}
}see \eqref{eq:K2}, \eqref{eq:K13j}, \eqref{eq:K24j} and Figure \ref{fig:charts}, provide an atlas of $\overline{\overline{\mathbb P}}$. Diffeomorphic change of coordinate maps between overlapping coordinate charts can be derived using \eqref{eq:K2}, \eqref{eq:K13j_2} and \eqref{eq:K24j_2}. For simplicity, we only show the changes of coordinates $\kappa_{2,ij}\,:\,K_2\rightarrow K_{ij}$, from $K_2$ to $K_{ij}$, $i\in \{1,3\}$, $j\in \{2,4,5\}$:
\begin{align*}
\kappa_{2,i2}\,&:\,(\eta_2,u_2,v_2)\mapsto \begin{cases}
\eta_i \ =s_i \eta_2 u_2,\\
\rho_{i2} =s_i u_2^{-1},\\
v_{i2} =v_2,
\end{cases}\quad v_2\in \mathbb R,\quad s_i u_2>0,\\
\kappa_{2,i4}\,&:\,(\eta_2,u_2,v_2)\mapsto\begin{cases}
\eta_i \ =s_i \eta_2 u_2,\\
\rho_{i4}=-s_i v_2 u_2^{-1},\\
\sigma_{i4} =-v_2^{-1},
\end{cases}\quad v_2<0,\quad s_i u_2>0,\\
\kappa_{2,i5}\,&:\,(\eta_2,u_2,v_2)\mapsto \begin{cases}
\eta_i \ =s_i \eta_2 u_2,\\
\rho_{i5} =s_i v_2 u_2^{-1},\\
\sigma_{i5} =v_2^{-1}
\end{cases}\quad v_2>0,\quad s_i u_2>0.
\end{align*}
Here we have also emphasized where these mappings are well-defined. 

\subsection{\SJJ{Analysis in the charts $K_{ij}$ used to cover the blow-up of $\overline{\Sigma}_b$}}\label{sec:K13j}

\SJJ{We now consider the equations in charts $K_{ij}$ with $(i,j)\in \{1,3\}\times \{2,4,5\}$, which cover the blow-up of $\overline{\Sigma}_b$.} The equations in the charts $K_{i2}$ are given by
\begin{equation}
\label{eq:bar_sigma_eqns_cylinder}
\begin{split}
r_a' &= {\phi}(v_{i2}) - r_a , \\
r_b' &= 1-{\phi}(s_i \rho_{i2}^{-1}) - \gamma r_b , \\
\eta_i' &= \mu s_i \eta_i \rho_{i2} G_a (r_a, s_i \eta_i) , \\
v_{i2}' &= \mu  G_b(r_b, \rho_{i2} \eta_i v_{i2}) , \\
\rho_{i2}' &= - \mu s_i \rho_{i2}^2 G_a(r_a, s_i \eta_i) ,
\end{split}
\end{equation}
with $\eta_i\ge 0$, $\rho_{i2}\ge 0$.
Both of the systems defined by \eqref{eq:bar_sigma_eqns_cylinder} (corresponding to $i\in \{1,3\}$)  are smooth and slow-fast with respect to $\mu\to 0$. Notice in particular that ${\phi}(s_i \rho_{i2}^{-1})$, $\rho_{i2}>0$, extends smoothly to $\rho_{i2}=0$ since
\begin{align*}
\lim_{\rho_{i2}\to 0} {\phi}(s_i \rho_{i2}^{-1}) = \begin{cases}
0 & \text{if} \ i=1,\\
1 & \text{if} \ i=3.	                                                 \end{cases}
\end{align*}
Here we have also used \eqref{eq:Hpm}.
Therefore the local versions $\mathcal M_{b,i2}$ of $\overline{\overline{\mathcal M}}_{b}$, defined by
\begin{align*}
r_a = {\phi}(v_{i2}),\qquad r_b = \gamma^{-1} (1-{\phi}(s_i \rho_{i2}^{-1})),
\end{align*}
are  $3$-dimensional attracting and normally critical manifolds of \eqref{eq:bar_sigma_eqns_cylinder} for $\mu=0$ with non-trivial eigenvalues $-1$ and $-\gamma < 0$. 

\

The equations in the charts $K_{i4}$ and $K_{i5}$ are given by
\begin{equation}
\label{eq:bar_v_eqns_cylinder}
\begin{split}
r_a' &= {\phi}(m_j \sigma_{ij}^{-1}) - r_a , \\
r_b' &= 1-{\phi}(s_i (\rho_{ij} \sigma_{ij})^{-1}) - \gamma r_b , \\
\eta_i' &= \mu s_i \eta_i \rho_{ij} \sigma_{ij} G_a (r_a, s_i \eta_i) , \\
\rho_{ij}' &= \mu m_j \rho_{ij} \sigma_{ij} \left[ G_b(r_b, m_j \rho_{ij} \eta_i) - s_i m_j \rho_{ij} G_a(r_a, s_i \eta_i ) \right] , \\
\sigma_{ij}' &= - \mu m_j \sigma_{ij}^2  G_b(r_b, m_j \rho_{ij} \eta_i) ,
\end{split}
\end{equation}
with $\rho_{ij}\ge 0$, $\sigma_{ij}\ge 0$, $(i,j) \in \{1,3\}\times \{4,5\}$, and
\[
m_j = 
\begin{cases}
-1 & \text{if} \ j=4 , \\
1 & \text{if} \ j=5 .
\end{cases}
\]
The equations in \eqref{eq:bar_v_eqns_cylinder} define four different systems (since there are four distinct couples $\{i,j\}$), each of which are smooth slow-fast systems with respect to $\mu\to 0$.  Notice, similarly to above in $K_{i2}$, that ${\phi}(m_j\sigma_{ij}^{-1})$, $\sigma_{ij}>0$ and ${\phi}(s_i (\rho_{ij} \sigma_{ij})^{-1})$ extend smoothly (by \eqref{eq:Hpm}) to $\rho_{ij}=0$ and $\sigma_{ij}=0$. Therefore the local versions $\mathcal M_{b,ij}$ of $\overline{\overline{\mathcal M}}_{b}^\pm$, defined by 
\begin{align*}
r_a = {\phi}(m_j \sigma_{ij}^{-1}),\qquad 
r_b &= \gamma^{-1} (1-{\phi}(s_i (\rho_{ij} \sigma_{ij})^{-1})),
\end{align*}
are $3$-dimensional attracting and normally critical manifolds of \eqref{eq:bar_sigma_eqns_cylinder} for $\mu=0$ with non-trivial eigenvalues $-1$ and $-\gamma < 0$.

\subsection{\SJJ{Analysis in the charts $K_{ij}$ used to cover the blow-up of $\overline{\Sigma}_a$}}\label{sec:K45j}

\SJJ{Finally, we consider the equations in charts $K_{ij}$ with $(i,j)\in \{4,5\}\times \{1,2,3\}$, which cover the blow-up of $\overline{\Sigma}_a$.} The equations in the charts $K_{i2}$ are given by
\begin{equation}
\label{eq:bar_sigma_eqns_cylinder_2}
\begin{split}
r_a' &= {\phi}(s_i \rho_{i2}^{-1}) - r_a , \\
r_b' &= 1-{\phi}(u_{i2}) - \gamma r_b , \\
u_{i2}' &= \mu G_a(r_a, \eta_i \rho_{i2} u_{i2}) , \\
\eta_i' &= \mu s_i \eta_i \rho_{i2} G_b (r_b, s_i \eta_i) , \\
\rho_{i2}' &= - \mu s_i \rho_{i2}^2  G_b(r_b, s_i \eta_i) ,
\end{split}
\end{equation}
with $\eta_i\ge 0$, $\rho_{i2}\ge 0$.
Both of the systems defined by \eqref{eq:bar_sigma_eqns_cylinder_2} are smooth and slow-fast with respect to $\mu\to 0$.  Notice, as above\SJ{, that} ${\phi}(s_i \rho_{i2}^{-1})$, $\rho_{i2}>0$, extends smoothly to $\rho_{i2}=0$ since
\begin{align*}
\lim_{\rho_{i2}\to 0} {\phi}(s_i \rho_{i2}^{-1}) = \begin{cases}
0 & \text{if} \ i=4,\\
1 & \text{if} \ i=5.	                                                 \end{cases}
\end{align*}
Here we have again used \eqref{eq:Hpm}.
Therefore the local versions $\mathcal M_{a,i2}$ of $\overline{\overline{\mathcal M}}_{a}$, defined by
\begin{align*}
r_a = {\phi}(s_i \rho_{i2}^{-1}),\qquad r_b = \gamma^{-1} (1-{\phi}(u_{i2})),
\end{align*}
are $3$-dimensional attracting and normally critical manifolds of \eqref{eq:bar_sigma_eqns_cylinder_2} for $\mu=0$ with non-trivial eigenvalues $-1$ and $-\gamma < 0$. 

The equations in the phase-directional charts $K_{i1}$ and $K_{i3}$ are given by
\begin{equation}
\label{eq:bar_u_eqns_cylinder}
\begin{split}
r_a' &= {\phi}(s_i ( \rho_{ij} \sigma_{ij})^{-1}) - r_a , \\
r_b' &= 1-{\phi}(m_j \sigma_{ij}^{-1}) - \gamma r_b , \\
\rho_{ij}' &= \mu m_j \rho_{ij} \sigma_{ij} \left[ G_a(r_a, m_j \rho_{ij} \eta_i) - s_i m_j \rho_{ij} G_b(r_b, s_i \eta_i ) \right] , \\
\eta_i' &= \mu s_i \eta_i \rho_{ij} \sigma_{ij} G_b (r_b, s_i \eta_i) , \\
\sigma_{ij}' &= - \mu m_j \sigma_{ij}^2 G_a(r_a, m_j \rho_{ij} \eta_i) .
\end{split}
\end{equation}
with $\rho_{ij}\ge 0$, $\sigma_{ij}\ge 0$,
$(i,j) \in \{4,5\}  \times \{1,3\}$, and
\[
m_j = 
\begin{cases}
-1 & \text{if} \ j=1 , \\
1 & \text{if} \ j=3 .
\end{cases}
\]
The equations in \eqref{eq:bar_u_eqns_cylinder} define four different systems (there are four distinct couples $\{i,j\}$), each of which are smooth and slow-fast with respect to $\mu\to 0$. Notice again that ${\phi}(s_i ( \rho_{ij} \sigma_{ij})^{-1})$ and ${\phi}(m_j \sigma_{ij}^{-1})$ both extend smoothly (by \eqref{eq:Hpm}) to $\rho_{ij}=0$ and $\sigma_{ij}=0$. Therefore the local versions $\mathcal M_{a,ij}$ of $\overline{\mathcal M}_{a}^\pm$ defined by
\begin{align*}
r_a = {\phi}(s_i ( \rho_{ij} \sigma_{ij})^{-1}),\qquad r_b = \gamma^{-1} (1-{\phi}(m_j \sigma_{ij}^{-1})), 
\end{align*}
are $3$-dimensional attracting and normally hyperbolic critical manifolds of \eqref{eq:bar_u_eqns_cylinder} for $\mu=0$ with non-trivial eigenvalues $-1$ and $-\gamma<0$. 

\subsection{\rsp{Completing the proof of Assertion \ref{existence}}}\label{app:existence}

By working in the $13$ different charts \eqref{eq:allKs}, which cover $\overline{\overline{\mathbb P}}$, we have now shown that \eqref{eq:M0} defines an attracting and normally hyperbolic  critical manifold $\overline{\overline{\mathcal M}}\subset \overline{\overline{\mathbb P}}$; in particular, the nontrivial eigenvalues were $-1$ and $-\gamma<0$ in all charts. In this way, by fixing $k\in \mathbb N$ and a compact submanifold $\overline{\overline{\mathcal M}}_0\subset \overline{\overline{\mathcal M}}$, we obtain a $C^k$-smooth locally invariant \SJJ{and attracting} slow manifold $\overline{\overline{\mathcal M}}_\mu$ for all $0<\mu<\mu_0$ by \cite[Thm.~9.1]{Fenichel1979} \SJJ{(for some sufficiently small $\mu_0 > 0$). This} slow manifold is $C^k$ $\mathcal O(\mu)$-close to the unperturbed manifold $\overline{\overline{\mathcal M}}_0$ on $\overline{\overline{\mathbb P}}$ (in particular, in the $13$ different charts). Next, we notice that the blow-up maps $\overline \Phi$, $\overline{\overline{\Phi}}_b^\pm$ and $\overline{\overline{\Phi}}_a^\pm$ are diffeomorphism\SJ{s} for $\eta>0$ and $\rho>0$; in particular \eqref{eq:K2}, \eqref{eq:K13j_2} and \eqref{eq:K24j_2} define (local) diffeomorphisms for $\eta_i>0$ and $\rho_{ij}>0$. Consequently, upon applications of these maps, we obtain a `blown-down' version $\mathcal M_\mu$ of $\overline{\overline{\mathcal M}}_\mu \cap\{\eta>0,\rho>0\}$, which is locally invariant for \eqref{eq:main_extended} in the $(r_a,r_b,p_a,p_b,\sigma)$\SJ{-space} with $0<\sigma<\sigma_0$. Since $\sigma'=0$, it follows that $$\mathcal M_\mu=\bigcup_{\sigma\in (0,\sigma_0)} \mathcal M_\mu(\sigma) \times \{\sigma\},$$ \SJJ{where} $\mathcal M_\mu(\sigma)$ \SJ{is} locally invariant for \SJJ{system} \eqref{eq:main0} for all $0<\sigma<\sigma_0$. This completes the proof of assertion \ref{existence} \SJJ{in} Theorem \ref{thm:thm0}.

\section{Proof of Proposition \ref{prop:PWS_dynamics}}\label{app:prop}

\SJ{Assertions} (i)-(vii) follow from direct calculations using the equations in \eqref{eq:pwl_limit_reduced} and the definition of the (reduced) switching manifold $\widetilde \Sigma$ in \eqref{eq:switching_manifold_reduced}; the details are omitted for brevity. 

We therefore focus on (viii).  Consider first an initial condition $(p_a(0),p_b(0))=(1,p_b(0))$ with $p_b(0)>1$. 
It is then elementary (using the piecewise-linear structure of \eqref{eq:pwl_limit_reduced}) to show that there are \SJ{minimal `hitting times'} $0<T_1<T_2<T_3<T_4$ such that $p_b(T_1)=p_a(T_2)=p_b(T_3)=p_a(T_4)=1$ and $p_a(T_1)>1$, $p_b(T_2)<1$, $p_a(T_3)<1$, $p_b(T_4)>1$, see also \eqref{eq:poincareP} below. This shows that there is a well-defined Poincar\'e map on $\{p_a=1,\,p_b>1\}$ given by $P\,:\,p_b(0)\mapsto p_b(T_4)$ with $p_b(0),p_b(T_4)>1$. There can be no fixed-points of this diffeomorphism since the piecewise-linear vector field \SJJ{associated with system} \eqref{eq:pwl_limit_reduced} has constant divergence equal to $-(\delta+1)<0$. This follows from Dulac's theorem which is also applicable in the present context because there is no sliding along the discontinuity sets, only crossing\SJ{; see e.g.~\cite{DaCruz2020}}.

To complete the proof, we show that
$P$ extends $C^\infty$-smoothly to $p_b=1$ with 
\begin{align}
P(1)=1,\quad P'(1)=1 \quad \mbox{and}\quad P''(1)=-\frac{(\delta + 1) \xi_a \xi_b}{2\delta (\xi_b-\gamma)} <0.\label{eq:Pexpr}\end{align}
Indeed, since $P$ has no fixed-points, it follows that $1<P(p_b)<p_b$ for all $p_b>1$ and therefore  $P^n(p_b)\rightarrow 1^+$ for $n\rightarrow \infty$.
To prove \eqref{eq:Pexpr}, we simply solve the linear equations \SJ{to} obtain
\begin{align}
\begin{cases}
p_a(T_1)=\xi_a+(1-\xi_a) p_b(0)^{-\delta^{-1}},\\
p_b(T_2)=p_a(T_1)^{-\delta},\\
p_a(T_3)=\left(\frac{\xi_b-\gamma}{\xi_b-\gamma p_b(T_2)}\right)^{\delta^{-1}},\\
p_b(T_4) =  \xi_b \gamma^{-1} +(1-\xi_b \gamma^{-1})\left( \frac{\xi_a-1}{\xi_a-p_a(T_3)}\right)^\delta.
\end{cases} \label{eq:poincareP}
\end{align}
From these expressions, one may directly realize that $P$ has a smooth extension to $p_b(0)=1$ with $P(1)=1$ and $P'(1)=1$. A simple calculation also shows the expression for $P''(1)$.

\end{appendix}
	
%
\end{document}